\documentclass[a4paper,10pt]{amsart}

\usepackage[utf8]{inputenc}
\usepackage{amsmath}
\usepackage{amssymb}
\usepackage{graphicx}
\usepackage{overpic}
\usepackage{todonotes}
\usepackage{colonequals}

\usepackage[defernumbers=true,
            natbib=true,     
            maxbibnames=99,  
            giveninits=true, 
            sortcites=true,  
            backend=biber]{biblatex}
\usepackage{hyperref}

\newcommand{\bs}[1]{{\boldsymbol #1}}

\newcommand{\R}{\mathbb{R}}
\DeclareMathOperator*{\argmin}{arg\,min}

\newcommand{\be}{\mathbf{e}}

\newcommand{\bp}{\mathbf{p}}
\newcommand{\bu}{\mathbf{u}}

\newcommand{\bC}{\mathbf{C}}

\providecommand{\abs}[1]{\lvert#1\rvert}
\providecommand{\norm}[1]{\lVert#1\rVert}

\newtheorem{theorem}{Theorem}
\newtheorem{corollary}[theorem]{Corollary}
\newtheorem{lemma}[theorem]{Lemma}

\makeatletter
\def\paragraph{\@startsection{paragraph}{4}%
  \z@\z@{-\fontdimen2\font}%
  {\normalfont\bfseries}}
\makeatother

\graphicspath{{gfx/}}

\title[Solving primal plasticity problems quickly]
      {Solving primal plasticity increment problems in the time of a single predictor--corrector iteration}
\author{Oliver Sander}\thanks{The author would like to thank Carsten Gräser and Patrizio Neff for the
 helpful discussions.}
\address{Oliver Sander\\
Technische Universität Dresden\\
Institut für Numerische Mathematik\\
Zellescher Weg 12--14\\
01069 Dresden\\
Germany}
\email{oliver.sander@tu-dresden.de}

\begin{document}

\begin{abstract}
The Truncated Nonsmooth Newton Multigrid (TNNMG) method is a well-established method for the solution of strictly convex
block-separably nondifferentiable minimization problems.  It achieves multigrid-like performance even for non-smooth
nonlinear problems, while at the same time being globally convergent and without employing
penalty parameters.  We show that the algorithm can be applied to the primal problem of classical linear elastoplasticity
with hardening.  Numerical experiments show that the method is considerably faster
than classical predictor--corrector methods.  Indeed, solving an entire increment problem with
TNNMG takes less time than a single predictor--corrector iteration for the same problem.
Since the algorithm does not rely on differentiability of the objective functional, nonsmooth
yield functions like the Tresca yield function can be easily incorporated.
The method is closely related to a predictor--corrector scheme with a consistent tangent predictor and
line search.  We explain the
algorithm, prove global convergence, and show its efficiency using a standard benchmark from the literature.
\end{abstract}

\maketitle

\section{Introduction}

The equations of linear plasticity with hardening are one of the classic problems in numerical
analysis~\cite{simo_hughes:1998,han_reddy:2013}.
The main challenge is the nondifferentiability of the flow rule.  Plastic flow only happens on the
boundary of the elastic region, and its onset is nonsmooth.  While the elastic region itself has
a smooth boundary for several important yield laws, more general models like the Tresca yield
model also involve elastic regions whose boundaries are only piecewise smooth.

The equations of linear plasticity exist in two forms.  The classical form considers displacements and stresses as
primary unknowns, and the yield function appears directly in the equation.  This form, nowadays called the
{\em dual formulation}~\cite{han_reddy:2013}, is a mixed variational inequality.
Algorithms for this formulation are typically of predictor--corrector type:  Solving an elastic problem for the
displacements (predictor) alternates with solving a nonsmooth problem for the stresses
(corrector)~\cite{simo_hughes:1998}.  The second step is not
expensive, as the nonsmooth part of the equation decouples into pointwise contributions.
The predictor step, however, involves solving one global linear system of equations at each iteration.

Increased insight into the convex-analytic foundations of linear plasticity has led to rising interest in the
\emph{primal formulation}~\cite{han_reddy:2013}.
This formulation results from considering displacement and plastic strain as the
primary unknowns.  Unlike the dual formulation, the increment problem of primal linear plasticity
can be written as a strictly convex minimization problem, where the nondifferentiability now appears
in form of the support function of the elastic region, the so-called
\emph{dissipation function}.  By standard results from convex analysis, this function is always convex and
positively homogeneous, but it has always points of nondifferentiability~\cite[Chap.\,8.E]{rockafellar_wets:2010}.

The energy of such increment problems is block-separably nonsmooth, by which we mean that it has the form
\begin{equation}
\label{eq:introduction_functional}
 J : \R^N \to \R \cup \{ \infty \},
 \qquad
 J(x) = J_0(x) + \sum_{i=1}^m \varphi_i(R_ix),
\end{equation}
where $J_0:\R^N \to \R$ is convex, coercive, and twice continuously differentiable, and
the $\varphi_i : \R^{N_i} \to \R \cup \{ \infty \}$, $i=1,\dots,m$ (which contain the dissipation function)
are convex, proper, and lower semi-continuous. The $R_i : \R^N \to \R^{N_i}$ are restriction operators,
$\sum_{i=1}^m N_i = N$, and we implicitly identified $\R^N$ with $\prod_{i=1}^m \R^{N_i}$.
This minimization structure provides ways to construct new classes of solver algorithms.
One common approach is to use predictor--corrector type methods, alternating elastic problems for the displacement
with pointwise nonlinear problems for the plastic
strain~\cite{caddemi_martin:1991,martin_caddemi:1994,han_reddy:2013}.
While these algorithms were presented for von Mises plasticity, they are easily extendable
to nonsmooth dissipation functions.

On the other hand, little convergence theory appear in the literature for these methods.  \citet{caddemi_martin:1991}
\cite{martin_caddemi:1994}
prove monotonic convergence, i.e., decrease of the energy at each iteration, but this does not
imply convergence of the solver itself.  \citet{carstensen:1997} obtains global convergence
for the predictor--corrector method with an elastic predictor, but it is generally agreed upon
that this choice of predictor is not as efficient as its competitors.

A second class of algorithms exploits the fact that for the von Mises dissipation function and isotropic
elasticity, a formula for the minimizer with respect to a single plastic strain variable is
available~\cite[Prop.\,7.1]{alberty_carstensen_zarrabi:1999}.  This allows to eliminate the plastic strain
variables from the increment energy functional.  By the Moreau theorem, the resulting functional is strictly convex
and even Fr\'echet differentiable, and can therefore be minimized by quasi-Newton~\cite{alberty_carstensen_zarrabi:1999}
or slant Newton methods~\cite{gruber_valdman:2009}.  Both \cite{alberty_carstensen_zarrabi:1999} and
\cite{gruber_valdman:2009} provide proofs of the local superlinear convergence of their respective methods.

In a paper closely related to ours,
\citeauthor{carstensen:1997} proved convergence of an exact multiplicative Schwarz methods~\cite{carstensen:1997}
(a nonlinear block Gauß--Seidel method in the terminology of this paper).
Under conditions slightly stronger than ours (uniform convexity of the smooth part of the functional),
he even showed \emph{linear} global convergence. As a special case of his general framework,
he obtains the above-mentioned convergence proof of the predictor--corrector method with an elastic
predictor. On the other hand, the proof assumes that the local minimization problems are solved exactly.
Also, the constant of the linear convergence rate depends on the number of subspaces, which makes the
result highly mesh-dependent.


\bigskip

All these methods are expensive in the sense that an entire linear system of equations has to be solved at each
iteration.  The Schwarz method of \citeauthor{carstensen:1997} is the exception that proves the rule,
but it converges only very slowly, as even Gauß--Seidel methods for linear problems do.

On the other hand, nonlinear multigrid methods exist that can solve block-separably nonsmooth convex  minimization
problems with the same efficiency as linear multigrid methods for linear problems.  This claim has been demonstrated
for contact problems~\cite{wohlmuth_krause:2003,graeser_sack_sander:2009},
phase-field models~\cite{graeser:2011,graeser_sack_sander:2009},
friction problems~\cite{pipping_sander_kornhuber:2015}, and others.
We focus here on the Truncated Nonsmooth Newton Multigrid (TNNMG) method originally introduced
by \citet{graeser_kornhuber:2009} for obstacle problems, but later generalized for minimization problems
for much more general block-separable problems~\cite{graeser_sander:2017}.
The method alternates a nonlinear block Gauß--Seidel iteration
with a particular linear correction problem and a line search.  Crucially, the linear correction
problems are not solved exactly.  Rather, only a single multigrid step is performed at each
iteration.  Geometric  multigrid works best, but an algebraic multigrid step can be used just
as well, when a grid hierarchy is not available.
The method can be shown to converge to the unique minimizer from any initial iterate~\cite{graeser_sander:2017}.
It does not do any regularization, and consequently there are no artificial
regularization or penalty parameters to choose.  Moreover, eventually the method degenerates to a
linear multigrid method on the subspace of inactive degrees of freedom, and hence converges with
multigrid efficiency.

The convergence proof in~\cite{graeser_sander:2017} makes very weak assumptions
on the local block Gauß--Seidel solvers. This allows to use very inexact solvers for the local
non-smooth problems. While the present paper focuses on von Mises plasticity, where exact
minimization formulas are available~\cite{alberty_carstensen_zarrabi:1999}, the TNNMG theory
encompasses much more general cases. We expect this to allow for a large speedup also on more general
linear plasticity problems. The framework also encompasses most other models of linear
associative primal plasticity, including different hardening models, smooth and nonsmooth
yield surfaces, and certain gradient and crystal plasticity models.
This will be the subject of later work.

\bigskip

In this paper, we focus on von Mises plasticity with kinematic hardening.
To keep the paper short we only hint at Tresca yielding, isotropic hardening,
and gradient plasticity (in Section~\ref{sec:gradient_plasticity},
the Aifantis model). A more detailed investigation of these models is left for future work.
In Chapter~\ref{sec:problem_formulation},
we revisit the equations of primal linear plasticity in their strong and weak forms.
In Chapter~\ref{sec:discrete_problem}, we discretize in time and space, and deduce the algebraic increment minimization
problem.  This is all standard.

In Chapter~\ref{sec:tnnmg_method} we then introduce the TNNMG method.  We first state its general form
and convergence results, and then apply them to the case of primal plasticity with hardening.
We discuss choices for the nonlinear smoother for von Mises and Tresca plasticity, and introduce
a trick to make the linear correction problems positive definite.
In Section~\ref{sec:predictor_corrector_methods}, we discuss the close relationship between
the TNNMG method and the predictor--corrector method with the consistent tangent predictor and
a line search \cite{martin_caddemi:1994}.
Finally, we give a glimpse of how the same method can be applied
to the seemingly much more challenging gradient-plasticity problem.

Chapter~\ref{sec:numerical_experiments} gives numerical experiments.  We use the plastic deformation of a square
with a hole, and we compare the TNNMG solver with the predictor--corrector method with a consistent tangent predictor.
On average,
the TNNMG method needs roughly twice to three times as many iterations per time step, but each iteration is much cheaper.

\section{Problem formulation}
\label{sec:problem_formulation}

We briefly revisit here the primal problem of linear plasticity with kinematic and/or isotropic hardening
as presented in~\cite{han_reddy:2013}.  We use von Mises and Tresca yield functions as examples,
but our approach also works for other smooth or nonsmooth associative flow laws.

\subsection{Strong formulation}

Consider a $d$-dimensional linearly elastoplastic object, with reference configuration given by a
domain $\Omega \subset \R^d$ with Lipschitz boundary $\Gamma$.  We denote by $\bu : \Omega \to \R^d$
the displacement field, and use the linear strain
\begin{equation*}
 \bs{\varepsilon}(\bs{u}) \colonequals \frac{1}{2}\big(\nabla \bs{u} + (\nabla \bs{u})^T \big)
\end{equation*}
to measure the change of shape. This strain $\bs{\varepsilon}$ is split additively
\begin{equation*}
 \bs{\varepsilon} = \mathbf{e} + \mathbf{p},
\end{equation*}
where $\mathbf{e}$ is the elastic strain, i.e., the strain that leads to stresses, and $\mathbf{p}$ is the plastic strain,
the one that tracks the irreversible deformation.  We will use $\mathbb{S}^d$ to denote the set of symmetric $d \times d$ matrices,
and $\mathbb{S}^d_0$ for the subset of symmetric matrices with vanishing trace.  The plastic strain $\mathbf{p}$
is a field of symmetric tensors, and it is assumed that $\mathbf{p}$ is pointwise trace-free,
$\mathbf{p} : \Omega \to \mathbb{S}^d_0$.
In the context of the primal formulation of plasticity, $\mathbf{u}$ and $\mathbf{p}$ are unknowns of the problem,
and $\bs{\varepsilon}$ and $\mathbf{e}$ are computed from $\mathbf{u}$ and $\mathbf{p}$.
In particular, this means that we will need initial conditions for $\mathbf{u}$ and $\mathbf{p}$.
For $\mathbf{u}$, we will additionally need standard Dirichlet and Neumann boundary conditions,
on relatively open subsets $\Gamma_D$ and $\Gamma_N$ of the boundary, respectively.

In linear plasticity, we assume that the stresses $\bs{\sigma}$ depend linearly on the elastic strain $\mathbf{e}$.
For a fixed plastic strain field, this means that $\bs{\sigma}$ is given by Hooke's law
\begin{equation*}
 \bs{\sigma} = \bs{C} : \bs{e} = \bs{C} : (\bs{\varepsilon}-\bs{p}),
\end{equation*}
where $\bs{C}$ is the forth-order elasticity tensor, with the usual symmetry properties~\cite[Chap.\,2.3]{han_reddy:2013}.
If the material is isotropic, which we assume for simplicity from now on, Hooke's law takes the form
\begin{equation}
\label{eq:stvenant_kirchhoff}
\bs{\sigma} = \lambda \operatorname{tr}(\bs{\varepsilon} - \bs{p})\mathbf{I} + 2\mu (\bs{\varepsilon} - \bs{p}),
\end{equation}
with $\mathbf{I}$ the $d \times d$ identity matrix.
The  scalar quantities $\lambda, \mu > 0$ are the Lam\'e constants, which describe the elastic properties
of the material.

Finally, the law of momentum conservation leads to the equilibrium equation
\begin{equation}
 \label{eq:equilibrium_equation}
 - \operatorname{div} \bs{\sigma} = \mathbf{f},
\end{equation}
where $\mathbf{f} : \Omega \to \R^d$ is an externally given volume force field.  In~\eqref{eq:equilibrium_equation},
we have omitted the inertial term, as we are only interested in quasistatic processes.  We consider this equation
on the time interval $[0,T]$.

\subsubsection{Flow rules}

The classical theory of plasticity postulates an open
region $\mathcal{E}$ in the space of stresses such that the object behaves purely elastically wherever
the stress $\bs{\sigma}$ is within that region.  The convexity of this region follows from the law of maximal plastic work
\cite[p.\,57]{han_reddy:2013}.
Plastic behavior happens only if $\bs{\sigma}$ is on the boundary $\mathcal{B}$ of $\mathcal{E}$.
This boundary is called the {\em yield surface}.  It is typically described
as the zero set of a scalar function $\phi$, called the \emph{yield function}
\begin{equation*}
 \mathcal{B}
 =
 \{ \bs{\sigma} \in \mathbb{S}^d \; : \; \phi(\bs{\sigma}) = 0 \}.
\end{equation*}

\paragraph{Example: The von Mises yield criterion}

The following frequently used yield criterion is based on experimental evidence that
the plastic flow does not depend on the spherical part of the stress.  Therefore, the von Mises yield criterion
only depends on the deviatoric stress
\begin{equation*}
 \bs{\sigma}^D \colonequals \bs{\sigma} - \frac{1}{d}(\operatorname{tr} \bs{\sigma}) \mathbf{I}.
\end{equation*}
In the von Mises theory, a material yields, i.e., exhibits plastic behavior, if the elastic shear energy density, which is a multiple of the invariant
 \begin{equation*}
  J_2(\bs{\sigma}^D)
  \colonequals
  \frac{1}{2} \norm{\bs{\sigma}^D}^2,
 \end{equation*}
reaches a critical value.  The resulting yield function is
\begin{equation}
\label{eq:von_mises_yield_function}
  \phi(\bs{\sigma})
  \colonequals
  \sqrt{J_2(\bs{\sigma}^D)} - \sigma_c
  =
  \norm{\bs{\sigma}^D} - \sigma_c.
\end{equation}
The critical value $\sigma_c$ for the yield function is obtained experimentally from the uniaxial yield stress
in tension $\sigma_0$.

\paragraph{Example: The Tresca yield criterion}

According to the Tresca yield criterion, a material flows plastically when the maximum shear stress
reaches the value $\frac{1}{2}\sigma_0$.  In principal stresses $\sigma_1, \dots, \sigma_d$, the maximum shear stress is given by
\begin{equation*}
 \frac{1}{2} \max_{i,j=1,\dots,d} \abs{ \sigma_i - \sigma_j}.
\end{equation*}
Consequently, the Tresca yield function is
\begin{equation*}
 \phi(\bs{\sigma})
 \colonequals
 \max_{i,j=1,\dots,d} \abs{ \sigma_i - \sigma_j} - \sigma_0.
\end{equation*}
Note that the corresponding yield surface $\mathcal{B} = \{ \bs{\sigma} \in \mathbb{S}^d : \phi(\bs{\sigma}) = 0 \}$ is not differentiable.

\bigskip

The elastic region $\mathcal{E}$ with boundary $\mathcal{B}$ also governs how the plastic strain $\mathbf{p}$ evolves with time.
Let $N_{\mathcal{E}}(\bs{\sigma})$ be the normal cone of $\mathcal{E}$
at $\bs{\sigma}$.  Then, by the principle of maximum plastic work~\cite[Chap.\,3.2]{han_reddy:2013} it follows that
the time evolution of $\mathbf{p}$ satisfies
\begin{equation}
\label{eq:flow_rule}
 \dot{\mathbf{p}} \in N_{\mathcal{E}}(\bs{\sigma}).
\end{equation}
This is equivalent to
\begin{equation*}
 \dot{\mathbf{p}} \in \lambda \partial \phi(\bs{\sigma}),
\end{equation*}
where $\partial \phi$ is the subdifferential of $\phi$ at $\bs{\sigma}$~\cite{rockafellar_wets:2010}.
The quantity $\lambda$ (not to be confused with the Lam\'e parameter of the same name) is called the {\em plastic multiplier}.

\subsubsection{Hardening}

The model as described so far is known as {\em perfect plasticity}.  More general models do not view the
elastic region as fixed, but allow it to evolve with time as well.  Such a behavior is called \emph{hardening}.
We include two simple hardening models here, but emphasize that our solver approach is not restricted to these models in particular.

Kinematic hardening supposes that the initial yield surface can move around in stress space.
Its current position is described by a symmetric $d \times d$ matrix $\bs{\alpha}$.
In contrast, isotropic hardening allows the elastic region to change its size over
time.  This change in size is tracked by a scalar nonnegative variable $\eta$.  To describe both effects in one formulation,
we introduce the generalized plastic strain
\begin{equation*}
 \mathsf{P} \colonequals(\mathbf{p}, \bs{\alpha}, \eta).
\end{equation*}
From thermodynamic considerations we know that there are conjugate variables to $\mathbf{p}$, $\bs{\alpha}$, and $\eta$
which we call {\em generalized stresses}, and we denote them by
\begin{equation*}
 \mathsf{\Sigma} \colonequals (\bs{\sigma}, \mathbf{a}, g).
\end{equation*}
If the hardening behavior is linear, then there are two nonnegative scalars $k_1$ and $k_2$ such
that
\begin{equation*}
 \mathbf{a} = - k_1 \bs{\alpha},
 \qquad
 g = - k_2 \eta
\end{equation*}
(more general linear relationships between $\mathbf{a}$ and $\bs{\alpha}$ are usually not considered).
Since $\eta$ is nonnegative we obtain that $g \le 0$.
A yield surface $\mathcal{B}$ in the space of stresses $\mathbb{S}^d$ depending on parameters $\bs{\alpha}$, $\eta$
can be written as a static surface in the space of generalized stresses.
A corresponding yield function takes the form
\begin{equation}
\label{eq:yield_function_in_generalized_stress_space}
\Phi(\bs{\sigma}, \mathbf{a}, g) = \phi(\bs{\sigma + \mathbf{a}}) + G(g) - \sigma_c,
\end{equation}
where $\phi$ is a yield function of perfect plasticity, and $G$ is a monotone increasing function satisfying $G(0) = 0$.
The elastic region in the space of generalized stresses will be denoted by
\begin{equation*}
 K
 \colonequals
 \{ \mathsf{\Sigma} \; : \; \Phi(\mathsf{\Sigma}) \le 0 \}.
\end{equation*}

The evolution of the additional variables is governed by extension of the flow law~\eqref{eq:flow_rule}.
With the new variables, it now takes the form
\begin{equation*}
\dot{\mathsf{P}} \in N_K (\mathsf{\Sigma}),
\end{equation*}
or, alternatively,
\begin{equation*}
\dot{\mathsf{P}} \in \lambda \partial \Phi(\mathsf{\Sigma}),
\end{equation*}
for a scalar plastic multiplier $\lambda$.
If, as in~\eqref{eq:yield_function_in_generalized_stress_space}, the yield function is such that
$\frac{\partial \Phi}{\partial \bs{\sigma}} = \frac{\partial \Phi}{\partial \bs{a}}$,
it turns out that $\mathbf{p}$ evolves exactly like $\bs{\alpha}$, and hence
the two variables are usually identified.

\subsubsection{Dissipation functions}
The abstract flow rule
\begin{equation}
\label{eq:abstract_flow_rule}
 \dot{\mathsf{P}} \in N_K(\mathsf{\Sigma})
\end{equation}
forms the basis of the dual formulation of plasticity.  To arrive at the primal formulation,
we need to write $\mathsf{\Sigma}$ as a function of $\dot{\mathsf{P}}$.  For this we introduce the dissipation function $D$
as the support function of the elastic set $K$, i.e.,
\begin{equation*}
 D(\mathsf{P}) \colonequals \sup_{\mathsf{\Sigma} \in \mathcal{E}} \{ \mathsf{P} \diamond \mathsf{\Sigma} \}
 \in \R \cup \{ \infty \},
\end{equation*}
with the scalar product
\begin{equation*}
 \mathsf{P} \diamond \mathsf{\Sigma}
 \colonequals
 \mathbf{p} : \bs{\sigma} + \bs{\alpha} : \mathbf{a} + \eta g.
\end{equation*}
As we have identified $\mathbf{p}$ with $\bs{\alpha}$, the generalized plastic strain is now
$\mathsf{P} = (\mathbf{p}, \eta)$, and we use the scalar product
\begin{equation*}
 \mathsf{P} \diamond \mathsf{\Sigma}
 \colonequals
 \mathbf{p} : \bs{\sigma} + \eta g.
\end{equation*}

The following standard result is crucial: It allows that methods of convex optimization can be used
to solve primal plasticity problems.
\begin{lemma}[{\cite[Thm.\,8.24]{rockafellar_wets:2010}}]
\label{lem:dissipation_function_properties}
The function $D$ is convex, positively homogeneous, lower semicontinuous, and proper.
\end{lemma}
With this, the abstract flow law~\eqref{eq:abstract_flow_rule} can be solved for the generalized stresses $\mathsf{\Sigma}$
\begin{equation}
\label{eq:abstract_primal_flow_rule}
 \mathsf{\Sigma} \in \partial D(\dot{\mathsf{P}}),
\end{equation}
see~\cite[Lem.\,4.2]{han_reddy:2013}.

\paragraph{Example: The von Mises dissipation function \cite[p.\,113]{han_reddy:2013}}

Consider first the case of kinematic hardening only, i.e., $k_2=0$.  Then, combining~\eqref{eq:von_mises_yield_function}
with~\eqref{eq:yield_function_in_generalized_stress_space} yields the admissible set
\begin{equation*}
 K
  =
 \big\{ (\bs{\sigma}, \mathbf{a}) \; : \; \norm{\bs{\sigma}^D + \mathbf{a}^D} - \sigma_c \le 0 \big\}.
\end{equation*}
We introduce the new variable $\overline{\mathbf{a}} \colonequals \bs{\sigma} + \mathbf{a}$, and rewrite
the admissible set as
\begin{equation*}
 K
  =
 \big\{ \overline{\mathbf{a}} \; : \; \norm{\overline{\mathbf{a}}^D} - \sigma_c \le 0 \big\}.
\end{equation*}
The corresponding dissipation function is
\begin{align}
 \nonumber
 \mathbf{p} \in \mathbb{S}_0^d,
 \qquad
 D(\mathbf{p})
 & =
 \sup \big \{ \overline{\mathbf{a}} : \mathbf{p} \; : \; \norm{\overline{\mathbf{a}}^D} \le \sigma_c \big \} \\
 \nonumber
 & =
 \sup \big \{ \overline{\mathbf{a}}^D : \mathbf{p} \; : \; \norm{\overline{\mathbf{a}}^D} \le \sigma_c \big \} \\
 \nonumber
 & =
 \sup \big \{ \norm{\overline{\mathbf{a}}^D} \norm{\mathbf{p}} \; : \; \norm{\overline{\mathbf{a}}^D} \le \sigma_c \big \} \\
\label{eq:von_mises_dissipation_no_isotropic_hardening}
 & =
 \sigma_c \norm{\mathbf{p}}.
\end{align}

For the case with isotropic hardening, and with the assumption $G(g) = g$, the admissible set is
\begin{align*}
 K
 & =
 \big\{ \mathsf{\Sigma} \; : \; \phi(\mathsf{\Sigma}) = \phi(\bs{\sigma} + \mathbf{a}) + g - \sigma_c \le 0 \big\}.
\end{align*}
Inserting the definition~\eqref{eq:von_mises_yield_function} of the von Mises yield function and the variable
substitution $\overline{\mathbf{a}} \colonequals \bs{\sigma} + \mathbf{a}$ gives
\begin{align*}
 K
 & =
 \big\{ (\overline{\mathbf{a}}, g) \; : \; \norm{\overline{\mathbf{a}}^D} + g - \sigma_c \le 0, \; g \le 0 \big\}.
\end{align*}
To compute the dissipation function for this case, we write
\begin{align*}
 D(\mathbf{p}, \eta)
 & =
 \sup \big\{ \bar{\mathbf{a}} : \mathbf{p} + g \eta
    \; : \;
    \norm{\bar{\mathbf{a}}^D} \le \sigma_c - g, \quad g \le 0 \big\} \\
 & =
 \sup_{g \le 0} \Big \{ \sup_{\bar{\mathbf{a}} \in \mathbb{S}^d} \big\{ \bar{\mathbf{a}} : \mathbf{p}
     \; : \; \norm{\bar{\mathbf{a}}^D} \le \sigma_c - g \big\} + g\eta \Big\}.
\end{align*}
From~\eqref{eq:von_mises_dissipation_no_isotropic_hardening} we know that the inner $\sup$
is simply $(\sigma_c - g)\norm{\mathbf{p}}$.  Using this we get
\begin{align*}
 D(\mathbf{p}, \eta)
 & =
 \sup_{g \le 0} \big \{ (\sigma_c - g)\norm{\mathbf{p}} + g\eta \big\} \\
 & =
 \sup_{g \le 0} \big \{ \sigma_c \norm{\mathbf{p}} + g (\eta - \norm{\mathbf{p}}) \big\},
\end{align*}
from which we deduce that
\begin{equation}
\label{eq:von_mises_dissipation_with_isotropic_hardening}
 (\mathbf{p}, \eta) \in \mathbb{S}_0^d \times \R_+,
 \qquad
 D(\mathbf{p}, \eta)
 =
 \begin{cases}
  \sigma_c\norm{\mathbf{p}}  & \text{if $\norm{\mathbf{p}} \le \eta$}, \\
  + \infty & \text{otherwise}.
 \end{cases}
\end{equation}
Note that, unlike in the case without isotropic hardening, the dissipation function now takes values
in the extended real line $\R \cup \{ \infty \}$.

\paragraph{Example: The Tresca dissipation function}
If $k_2 = 0$, i.e., there is only kinematic hardening, the admissible set for the Tresca yield criterion is
\begin{equation*}
 K
  =
 \big\{ \bar{\mathbf{a}} \in \mathbb{S}^d \; : \; \max_{i,j} \,\abs{\bar{a}_i - \bar{a}_j} - \sigma_c \le 0 \big\}.
\end{equation*}
Here, $\bar{a}_i$ denotes the $i$-th eigenvalue of the symmetric matrix $\bar{\mathbf{a}}$.
For the relevant cases $d=2$ and $d=3$, the support function for this set is the scaled spectral radius
\begin{equation*}
D(\mathbf{p})
=
\sigma_c \rho(\mathbf{p})
\colonequals
\sigma_c \max \big\{ \abs{p_1}, \dots, \abs{p_d} \big\},
\end{equation*}
where $p_1, \dots, p_d$ are the eigenvalues of $\mathbf{p}$.  Since we could not find a proof for this
result in the literature, we give our own in the appendix of this manuscript.

With additional isotropic hardening, the admissible set is
\begin{equation*}
 K
  =
 \big\{ (\bar{\mathbf{a}},g)
      \; : \; \max_{i,j} \,\abs{\bar{a}_i - \bar{a}_j} +g - \sigma_c \le 0 \big\}.
\end{equation*}
Using the same argument as for the von Mises flow rule,
the dissipation function for this is (again only for $d=2,3$)
\begin{equation}
\label{eq:tresca_dissipation_with_isotropic_hardening}
 (\mathbf{p}, \eta) \in \mathbb{S}_0^d \times \R_+,
 \qquad
 D(\mathbf{p}, \eta)
 =
 \begin{cases}
  \sigma_c \rho(\mathbf{p})  & \text{if $\rho(\mathbf{p}) \le \eta$}, \\
  + \infty & \text{else}.
 \end{cases}
\end{equation}

\subsection{Weak formulation}
For the weak formulation we first specify the function spaces. Displacements $\mathbf{u}$ will be taken from the space
\begin{equation*}
V
\colonequals
(H_{\Gamma}^1(\Omega))^d
=
\big\{\bs{v} \in H^1(\Omega,\R^d)\; : \; \bs{v}_{|\Gamma_D}=0\big\},
\end{equation*}
where $\Gamma_D$ is the Dirichlet boundary.  For simplicity we assume that only homogeneous Dirichlet values occur.
To define the space of plastic strains we introduce
\begin{equation*}
Q_0
\colonequals
L^2(\Omega,\mathbb{S}^d_0).
\end{equation*}
Finally, for the isotropic hardening parameter we need
\begin{equation*}
M=L^2(\Omega).
\end{equation*}
We combine these spaces to the product space
\begin{equation*}
W=V\times Q_0 \times M.
\end{equation*}

\subsubsection{Weak formulation}

Recall the flow law of elastoplasticity~\eqref{eq:abstract_primal_flow_rule}
\begin{equation*}
\mathsf{\Sigma} \in \partial D(\dot{\mathsf{P}}), \quad \mathsf{\Sigma} = (\bs{\sigma}, \mathbf{a}, g).
\end{equation*}
By Lemma~\ref{lem:dissipation_function_properties}, $D$ is convex; in other words if
$(\dot{\bs{p}},\dot{\bs{\alpha}}, \dot{\eta}) \in \operatorname{dom} D$, then
\begin{align*}
D(\tilde{\bs{p}},\tilde{\eta}) & \ge D(\dot{\bs{p}}, \dot{\eta}) + \bs{\sigma} : (\tilde{\bs{p}}
- \dot{\bs{p}}) + \mathbf{a} : (\tilde{\bs{\alpha}} - \dot{\bs{\alpha}}) + g (\tilde{\eta} - \dot{\eta})
\end{align*}
for all
\begin{align*}
(\tilde{\mathbf{p}}, \tilde{\bs{\alpha}}, \tilde{\eta}) \in \mathbb{S}^d_0 \times \mathbb{S}^d_0 \times \R.
\end{align*}
Using $\mathbf{a} = -k_1 \bs{\alpha}$, $g = -k_2\eta$ and $\bs{\alpha} = \bs{p}$, this simplifies to
\begin{align*}
D(\tilde{\bs{p}},\tilde{\eta})
& \ge
D(\dot{\bs{p}}, \dot{\eta}) + (\bs{\sigma} - k_1 \bs{p}) : (\tilde{\bs{p}} -\dot{\bs{p}})
 - k_2 \eta (\tilde{\eta} - \dot{\eta})
\qquad
\forall (\tilde{\mathbf{p}}, \tilde{\eta}) \in \mathbb{S}^d_0 \times \R.
\end{align*}
Integrating over $\Omega$, and using the elastic law
$\bs{\sigma} = \bC (\bs{\varepsilon}(\mathbf{u})-\mathbf{p})$ we get
\begin{multline}
\label{eq:integrated_convexity}
\int_\Omega D(\tilde{\bs{p}}, \tilde{\eta})\,dx
\ge
\int_\Omega D(\dot{\bs{p}}, \dot{\eta})\,dx \\
  +
\int_\Omega \Big[ \bC (\bs{\varepsilon}(\mathbf{u})-\mathbf{p}) : (\tilde{\bs{p}} - \dot{\bs{p}})
- k_1 \bs{p} : (\tilde{\bs{p}} - \dot{\bs{p}}) - k_2 \eta (\tilde{\eta} - \dot{\eta}) \Big]\,dx
\qquad
\forall (\tilde{\mathbf{p}}, \tilde{\eta}) \in Q_0 \times M.
\end{multline}

Then, take the scalar product of $- \operatorname{div} \bs{\sigma} = \mathbf{f}$ with $\tilde{\mathbf{u}} - \dot{\mathbf{u}}$
for any $\tilde{\mathbf{u}} \in V$, integrate over $\Omega$ and use $\bs{\sigma} = \bC (\bs{\varepsilon}(\mathbf{u})-\mathbf{p})$
to obtain
\begin{equation}
\label{eq:weak_elasticity}
 \int_\Omega \bC(\bs{\varepsilon}(\mathbf{u}) - \mathbf{p}) : ( \bs{\varepsilon}(\tilde{\mathbf{u}}) - \bs{\varepsilon}(\dot{\mathbf{u}}))\,dx
 =
 \int_\Omega \mathbf{f} \cdot(\tilde{\mathbf{u}} - \dot{\mathbf{u}})\,dx
 \qquad
 \forall \tilde{\mathbf{u}} \in V.
\end{equation}
Finally, add~\eqref{eq:integrated_convexity} and~\eqref{eq:weak_elasticity} to obtain the following variational inequality:
Find $\mathbf{w}=(\bs{u}, \bs{p}, \eta) : [0,T] \to W$, $\mathbf{w}(\bs{0})=\bs{0}$,
so that for all $t \in (0,T)$ we have $\dot{\mathbf{w}}(t) \in W$ and $j(\dot{\mathbf{w}}(t)) < \infty$, and
\begin{equation}
\label{eq:variational_inequality}
a(\mathbf{w}(t), \tilde{\mathbf{w}} - \dot{\mathbf{w}}(t)) + j(\tilde{\mathbf{w}}) - j(\dot{\mathbf{w}}(t))
\ge
\langle l(t), \tilde{\mathbf{w}} - \dot{\mathbf{w}}(t) \rangle
\quad \forall \tilde{\mathbf{w}} \in W.
\end{equation}
In this variational inequality we have used the notation
\begin{align}
\label{eq:bilinear_form}
a(\mathbf{w},\tilde{\mathbf{w}})
 & \colonequals \int_\Omega \Big[\bs{C}(\bs{\varepsilon}(\bs{u}) - \bs{p}) : (\bs{\varepsilon}(\tilde{\bs{u}}) - \tilde{\bs{p}})
   + k_1 \bs{p} : \tilde{\bs{p}} + k_2 \eta \tilde{\eta} \Big]\,dx,\\
\nonumber
j(\tilde{\mathbf{w}}) & \colonequals \int_\Omega D(\tilde{\bs{p}}, \tilde{\eta})\,dx,
\end{align}
and the right hand side $l$ is given by
\begin{equation*}
\langle l(t), \tilde{\mathbf{w}} \rangle = \int_\Omega \bs{f}(t) \cdot \tilde{\bs{u}}\,dx.
\end{equation*}
Existence of a unique solution follows from Theorem~7.3 in~\cite{han_reddy:2013}, provided the bilinear form
is elliptic.  This is the content of the following result.
\begin{theorem}
\label{thm:ellipticity_of_the_bilinear_form}
 Let $k_1,k_2 >0$.
 Then the bilinear form $a(\cdot,\cdot)$ defined in~\eqref{eq:bilinear_form} is elliptic in the sense that
 \begin{equation*}
  a(\mathbf{w}, \mathbf{w})
  \ge
  c_0 (\norm{\mathbf{u}}^2_V + \norm{\mathbf{p}}^2_Q + \norm{\eta}^2)
 \end{equation*}
 for all $\mathbf{w} = (\mathbf{u}, \mathbf{p}, \eta) \in W$ such that $\mathbf{w}$ is
 in the effective domain $\operatorname{dom} D$ of the dissipation function almost
 everywhere in $\Omega$.
\end{theorem}
Corresponding results hold for the cases with only kinematic or only isotropic hardening
\cite[Chap.\,7.1]{han_reddy:2013}.

\section{Discrete problem}
\label{sec:discrete_problem}

\subsection{Time discretization}

We now divide the time interval $[0, T]$ into $N$ equal subintervals of size $\Delta t = T/N$, with time points $t_n$,
$0 \leq n \leq N$.  We denote by $\mathbf{w}_n \in W$ an approximation of the solution $\mathbf{w}$ at time $t_n$, and set
$\Delta \mathbf{w}_n \colonequals \mathbf{w}_n - \mathbf{w}_{n-1}$.

The variational inequality~\eqref{eq:variational_inequality} is discretized using a backward Euler scheme,
that is, derivatives at time $t_n$ are approximated by $\dot{\mathbf{w}} \approx \frac{\mathbf{w}_n - \mathbf{w}_{n-1}}{\Delta t}$.
(Other time discretization  methods can be handled similarly.)
Plugging this into~\eqref{eq:variational_inequality} yields
\begin{equation*}
a\Big(\mathbf{w}_n, \tilde{\mathbf{w}} - \frac{\Delta \mathbf{w}_n}{\Delta t} \Big) + j(\tilde{\mathbf{w}}) - j\Big(\frac{\Delta \mathbf{w}_n}{\Delta t}\Big)
\ge
\Big\langle l_n, \tilde{\mathbf{w}} - \frac{\Delta \mathbf{w}_n}{\Delta t} \Big\rangle \quad \forall \tilde{\mathbf{w}} \in W.
\end{equation*}
We use the positive homogeneity of $j$ to get rid of the $\Delta t$:
\begin{equation*}
a(\mathbf{w}_n, \tilde{\mathbf{w}} - \Delta \mathbf{w}_n) + j(\tilde{\mathbf{w}}) - j(\Delta \mathbf{w}_n)
\ge
\langle l_n, \tilde{\mathbf{w}} - \Delta \mathbf{w}_n \rangle
\quad
\forall \tilde{\mathbf{w}} \in W,
\end{equation*}
and rewrite this as a problem for the increment $\Delta \mathbf{w}_n$
\begin{multline}
\label{eq:increment_variational_inequality}
a(\Delta \mathbf{w}_n, \tilde{\mathbf{w}} - \Delta \mathbf{w}_n) + j(\tilde{\mathbf{w}}) - j(\Delta \mathbf{w}_n) \\
\ge
\langle l_n, \tilde{\mathbf{w}} - \Delta \mathbf{w}_n \rangle
- a(\mathbf{w}_{n-1}, \tilde{\mathbf{w}} - \Delta \mathbf{w}_n)
\quad
\forall \tilde{\mathbf{w}} \in W.
\end{multline}
Since $a(\cdot,\cdot)$ is symmetric, continuous, and elliptic on the set
$W_p \colonequals \{ w \in W \; : \; w \in \operatorname{dom} D\;\text{a.e.}\}$,
and $j(\cdot)$ is coercive, convex, and lower semicontinuous, we can
use a classic result from convex analysis~\cite[Chap.\,II, Prop.\,2.2]{ekeland_temam:1999} to obtain the following minimization formulation.
\begin{theorem}
 A function $\Delta \mathbf{w}_n \in W$ is a solution of the increment variational inequality~\eqref{eq:increment_variational_inequality}
 if and only if it minimizes the functional
 \begin{equation}
 \label{eq:continuous_energy_functional}
  \mathcal{L}(\mathbf{w})
  =
  \frac{1}{2}a( \mathbf{w},\mathbf{w}) + j(\mathbf{w}) - \langle l_n , \mathbf{w} \rangle + a(\mathbf{w}_{n-1},\mathbf{w}),
  \qquad
  \mathbf{w} = (\bs{u}, \bs{p}, \eta).
 \end{equation}
 Because of the ellipticity of $a(\cdot,\cdot)$,
 this functional is strictly convex, and coercive on $W$.
 Hence it has a unique minimizer.
\end{theorem}

In the following, we omit the subscript $n$ for simplicity.

\subsection{Space discretization}
\label{sec:space_discretization}

For a finite element approximation of the functional \eqref{eq:continuous_energy_functional},
we need approximations for the displacement space $V$, the space $Q_0$ of plastic strains, and the space $M$
of the isotropic hardening variable $\eta$.  Since only the displacements are ever differentiated with respect to space, it is most natural to choose
\begin{align*}
 V^h &   : \quad \text{first-order Lagrangian elements with values in $\R^d$}, \\
 Q_0^h & : \quad \text{zero-order elements with values in $\mathbb{S}_0^d$}, \\
 M^h   & : \quad \text{zero-order elements with values in $\R$}.
\end{align*}
Convergence error estimates for this choice can be found, e.g., in~\cite[Chap.\,12.1]{han_reddy:2013}
and~\cite{alberty_carstensen_zarrabi:1999,carstensen:1999}.
We briefly discuss the use of first-order elements for all three spaces
in Chapter~\ref{sec:gradient_plasticity} below.

To state the algebraic formulation of the increment minimization problem,
let $\{\phi_i\}_{i=1}^{n_1}$ and $\{\theta_i\}_{i=1}^{n_2}$ be the nodal bases of the scalar first-order and zero-order
Lagrangian finite element spaces, respectively.  We define the matrix $A$ and the vector $b$ as the stiffness matrix and load
vector corresponding to the bilinear form $a(\cdot,\cdot)$ defined in~\eqref{eq:bilinear_form}, and the functional
$\langle l, \cdot \rangle - a(\mathbf{w}_{n-1},\cdot)$, of~\eqref{eq:increment_variational_inequality}, respectively.
The matrix $A$ is symmetric and, under the conditions of
Theorem~\ref{thm:ellipticity_of_the_bilinear_form}, positive definite.

For the displacements we introduce the canonical $d$-valued basis
\begin{equation*}
 \{ \bs{\phi}_{i,j} \}_{i=1,\dots,n_1, j=1,\dots,d},
 \qquad
 \bs{\phi}_{i,j} = \phi_i \be_j, \quad i=1,\dots,n_1, j=1,\dots,d,
\end{equation*}
where the $\mathbf{e}_k$ are the canonical basis vectors in $\R^d$.  Representing the plastic strain
matrices $\bp$ is more involved.
Any $\mathbf{p}$ is symmetric and trace-free, and has therefore only
$d_p \colonequals d (d+1) / 2 -1$ independent degrees of freedom.  Any implementation should exploit this, to keep memory-
and run-time requirements low.
To this purpose we introduce a basis $\{B_i\}_{i=1}^{d_p}$ of the space $\mathbb{S}^d_0$,
and we represent any symmetric trace-free $d \times d$ matrix $\bp$ as a $d_p$-vector by
\begin{equation*}
 p \mapsto B(p) \colonequals \sum_{j=1}^{d_p} B_j p_j.
\end{equation*}
Then a function $\mathbf{p}_h$ in $Q_0^h$ can be represented as
\begin{equation*}
 \mathbf{p}_h(x)
 =
 \sum_{i=1}^{n_2} \theta_i(x) \sum_{j=1}^{d_p}  B_j p_{i,j},
 \qquad
 \forall x \in \Omega.
\end{equation*}
It is helpful if the map $B : \R^{d_p} \to \mathbb{S}_0^d$ is an isometry.
We achieve this by selecting appropriate basis vectors $B_1,\dots,B_{d_p}$.
For two-dimensional problems we use
\begin{equation*}
 B_1 = \frac{1}{\sqrt{2}}
 \begin{pmatrix}
  1 &   \\
    & -1
 \end{pmatrix},\quad
 B_2 = \frac{1}{\sqrt{2}}
\begin{pmatrix}
    & 1 \\
  1 &
 \end{pmatrix},
\end{equation*}
whereas for three-dimensional problems we use
\begin{multline*}
 B_1 = \frac{1}{\sqrt{2}}
 \begin{pmatrix}
  -1 &   & \\
    & 1 & \\
    &   & 0
 \end{pmatrix},\quad
 B_2 = \frac{1}{\sqrt{2}}
\begin{pmatrix}
    & 1  & \\
  1 &    & \\
    &   &
 \end{pmatrix},\quad
 B_3 = \frac{1}{\sqrt{2}}
 \begin{pmatrix}
  \frac{1}{\sqrt{3}} &   & \\
    & \frac{1}{\sqrt{3}} & \\
    &   & -\frac{2}{\sqrt{3}}
 \end{pmatrix},\\
 B_4 = \frac{1}{\sqrt{2}}
 \begin{pmatrix}
    &   & 1 \\
    &   & \\
   1 &  &
 \end{pmatrix},\quad
 B_5 = \frac{1}{\sqrt{2}}
 \begin{pmatrix}
    &   & \\
    &   & 1\\
    &  1 &
 \end{pmatrix}.
\end{multline*}
The isotropic hardening variable $\eta_h$ does  not pose any difficulties, since it is scalar-valued anyway.

\bigskip

Let $u_{i,j}$, $i=1,\dots,n_1$, $j=1,\dots, d$ be the coefficients of the displacement vector with respect
to the basis given above, and likewise for the plastic strain coefficients $p_{i,j}$, $i=1,\dots,n_2$,
$j=1,\dots,d_p$, and $\eta_i$, $i=1,\dots,n_2$.  Anticipating the use of the TNNMG nonsmooth multigrid
method in the following chapter, we arrange the coefficients in a particular order.
We first list all displacement coefficients
\begin{equation}
\label{eq:displacement_ordering}
 u_{11}, \dots, u_{1d}, u_{21},\dots, u_{2d}, \dots, u_{n_11},\dots, u_{n_1d},
\end{equation}
followed by the plastic strain and isotropic hardening coefficients grouped for each element
\begin{equation}
\label{eq:plastic_strain_ordering}
 p_{11},\dots,p_{1d_p}, \eta_1, \;p_{21}, \dots, p_{2d_p}, \eta_2, \dots,
 p_{n_21},\dots,p_{n_2d_p}, \eta_{n_2}.
\end{equation}
Simply omit the $\eta_i$ if isotropic hardening is not considered.

With this ordering, the stiffness matrix $A$ consists of $2 \times 2$ blocks
\begin{equation*}
 A
 =
 \begin{pmatrix}
  E & C \\
  C^T & P
 \end{pmatrix}.
\end{equation*}
The matrix entries of the block $E$ are
\begin{equation}
\label{eq:elasticity_matrix_entries}
 (E_{ij})_{kl} = \int_\Omega \bC (\bs{\epsilon}(\phi_i \be_k) : \bs{\epsilon}(\phi_j \be_l)\,dx,
 \qquad
 1 \le i,j \le n_1, \quad 1 \le k,l \le d.
\end{equation}
Matrix $P$ is block-diagonal.  Each block on the diagonal has size $(d_p+1) \times (d_p+1)$,
and the form
\begin{equation}
\label{eq:plastic_strain_matrix_entries}
P_{ii}
=
 \begin{pmatrix}
  \tilde{P}_{ii} & 0 \\
  0              & \tilde{p}_{ii}
 \end{pmatrix},
\end{equation}
where
\begin{equation*}
 (\tilde{P}_{ii})_{kl}
 =
 \int_\Omega \Big[ \bC\theta_i(x)^2 B_k : B_l + k_1 \theta_i(x)^2 B_k : B_l \Big] \,dx,
 \qquad
 1 \le i \le n_2, \quad 1 \le k,l \le d_p,
\end{equation*}
and
\begin{align*}
 \tilde{p}_{ii} = \int_\Omega k_2 \theta_i(x)^2\,dx.
\end{align*}
The off-diagonal blocks $C$, $C^T$ of $A$ contain interactions between the plastic strain and the
displacement.
Entries of $C$ have the form
\begin{equation}
\label{eq:coupling_matrix_entries}
 C_{ij}
 =
 \begin{pmatrix}
  \tilde{C}_{ij} & 0
 \end{pmatrix}
 \in
 \R^{d \times (d_p + 1)}.
\end{equation}
The last column is zero,
because the displacements $\bu$ couple to the isotropic hardening variable $\eta$ only via the
dissipation functional.  The left block of such a $C_{ij}$ is
\begin{multline*}
 (\tilde{C}_{ij})_{k,l}
 =
 - \int_\Omega \bC \bs{\epsilon}(\phi_i \be_k) : \theta_j(x) B_l\,dx\\
 1 \le i \le n_1, \quad 1 \le j \le n_2, \quad 1 \le k \le d, \quad 1 \le l \le d_p.
\end{multline*}

Restrict the increment functional $\mathcal{L}$ from~\eqref{eq:continuous_energy_functional}
to the finite element spaces $V^h$, $Q_0^h$, and $M^h$, and use the bases just introduced
to represent discrete functions by their coefficient vectors.  This yields the
algebraic minimization problem
\begin{align}
 \nonumber
 L & \; : (\R^d)^{n_1} \times \big( \R^{d_p} \times \R \big)^{n_2} \to \R \cup \{ \infty \} \\
 \label{eq:algebraic_plasticity_functional_unlumped}
 L(w)
 & \colonequals
 \frac{1}{2} w^T A w - b^T w
 +
 \int_\Omega D\bigg(\sum_{i=1}^{n_2} \theta_i(x) \sum_{j=1}^{d_p} B_j p_{i,j}, \sum_{i=1}^{n_2} \theta_i(x) \eta_i \bigg)\,dx,
\end{align}
where $D$ is a dissipation function, for example the von Mises function~\eqref{eq:von_mises_dissipation_with_isotropic_hardening} or the
Tresca dissipation function~\eqref{eq:tresca_dissipation_with_isotropic_hardening}.  Since the $\theta_i$ are piecewise constant,
and $D$ is positive homogeneous,
the sum over the grid elements can be moved out of the dissipation function, to obtain
\begin{align}
 \label{eq:algebraic_plasticity_functional}
 L(w)
 & \colonequals
 \frac{1}{2} w^T A w - b^T w
 +
 \sum_{i=1}^{n_2} \int_\Omega \theta_i(x)\,dx \cdot D\bigg( \sum_{j=1}^{d_p} B_j p_{i,j}, \eta_i\bigg).
\end{align}
The case without isotropic hardening is obtained from this by simply dropping the relevant rows and columns from $A$ and $b$,
and by replacing $D$ by a dissipation function without dependence on $\eta$.

\section{The Truncated Nonsmooth Newton Multigrid method}
\label{sec:tnnmg_method}

The Truncated Nonsmooth Newton Multigrid (TNNMG) method is designed to solve nonsmooth convex minimization problems
on Euclidean spaces $\R^N$ for functionals with the block-separability structure~\eqref{eq:introduction_functional}.
Let $\R^N$ be endowed with a block structure
\begin{equation*}
 \R^N = \prod_{i=1}^m \R^{N_i},
\end{equation*}
and call $R_i : \R^N \to \R^{N_i}$ the canonical restriction to the $i$-th block.
Typically, the factor spaces $\R^{N_i}$ will have small dimension, but the number of factors $m$
is expected to be large.  We denote by $V_k$ the subspace of $\R^N$ of all vectors that have zero
entries everywhere outside of the $k$-th block.
Assume that the objective functional $J:\R^N \to \R \cup \{ \infty \}$ is given by
\begin{align}
    \label{eq:min_problem}
    J(v) = J_0(v) + \sum_{i=1}^m \varphi_i (R_i v),
\end{align}
with a convex $C^2$ functional $J_0 :\R^N \to \R$, and convex, proper, lower semi-continuous (l.s.c.) functionals
$\varphi_i : \R^{N_i} \to \R \cup \{ \infty \}$. We will additionally assume that $J$ is strictly convex and coercive,
but these conditions can be relaxed~\cite{graeser_sander:2017}.

Given such a functional $J$, the iteration number $\nu \in \mathbb{N}_0$, and a given previous iterate $v^\nu \in \R^N$,
one iteration of the TNNMG method consists of the following four steps:
\begin{enumerate}
    \item\label{enum:presmoothing}
        \textbf{Nonlinear presmoothing}
        \begin{enumerate}
            \item
                Set $\tilde{v}^0 = v^\nu$
            \item
                For $k=1,\dots,m$ compute $\tilde{v}^k \in \tilde{v}^{k-1} + V_k$ as
                        \begin{align}
                        \label{eq:general_local_problem}
                            \tilde{v}^k \approx \argmin_{\tilde{v} \in \tilde{v}^{k-1} + V_k} J(\tilde{v})
                        \end{align}
            \item
                Set $v^{\nu+\frac12} = \tilde{v}^m$
        \end{enumerate}
    \item \label{enum:linear_correction}
        \textbf{Inexact linear correction}
        \begin{enumerate}
            \item
                Determine maximal subspace $W_\nu \subset \R^N$ such that
                $J|_{W_\nu}$ is $C^2$ at $v^{\nu + \frac{1}{2}}$
            \item
                Compute $c^\nu \in W_\nu$ as an inexact Newton step on $W_\nu$
                \begin{align}
                \label{eq:inexact_linear_correction}
                    c^\nu \approx
                        -\big(J''(v^{\nu+\frac12})|_{W_\nu \times W_\nu}\big)^{-1}
                        \big(J'(v^{\nu+\frac12})|_{W_\nu} \big)
                \end{align}
        \end{enumerate}
    \item \label{enum:projection}
        \textbf{Projection}\\
                Compute the Euclidean projection $c_\text{pr}^\nu = P_{\operatorname{dom}J - v^{\nu+1/2}}(c^\nu)$,
                i.e., choose $c_\text{pr}^\nu$ such that $v^{\nu + \frac{1}{2}} + c_\text{pr}^\nu$ is closest
                to $v^{\nu + \frac{1}{2}} + c^\nu$ in $\operatorname{dom} J$
    \item  \label{enum:line_search}
        \textbf{Damped update}
        \begin{enumerate}
            \item
                Compute a $\rho_\nu \in [0,\infty)$ such that
                $J(v^{\nu+\frac12} + \rho_\nu c_\text{pr}^\nu) \leq J(v^{\nu+\frac12})$
            \item Set $v^{\nu+1} = v^{\nu+\frac12} + \rho_\nu c_\text{pr}^\nu$
        \end{enumerate}
\end{enumerate}

The canonical choice for the linear correction step~\eqref{eq:inexact_linear_correction}
is a single linear multigrid step, which explains
why the overall method is classified as a multigrid method. If a grid hierarchy is available, then
a geometric multigrid method is preferable. Otherwise, an algebraic multigrid step will work just
as well.

For this method we have the following general convergence result~\cite{graeser_sander:2017}:
\newcommand{\ContinuousArgmin}{\tilde{\mathcal{M}}}
\newcommand{\InexactArgmin}{\mathcal{M}}
\newcommand{\Ccal}{\mathcal{C}}

\begin{theorem}
    \label{thm:gs_stationary}
    Let $J : \R^N \to \R \cup \{\infty\}$ be strictly convex, coercive, proper, lower-semicontinuous,
    and block-separably nonsmooth.
    Write $\InexactArgmin_k : \tilde{v}^{k-1} \mapsto \tilde{v}^k$ for the (inexact) minimization
    of $J$ in $\tilde{v}^{k-1} + V_k$, possibly by an iterative algorithm.
    For a correction operator $\Ccal:\operatorname{dom}J \to \operatorname{dom}J$, and an initial guess
    $v^0 \in \operatorname{dom} J$ let $v^\nu$ be given by the TNNMG algorithm.
    Assume that there are operators $\ContinuousArgmin_k$ of the form
    \begin{equation*}
    \ContinuousArgmin_k : \operatorname{dom} J \to \operatorname{dom}J,
    \qquad
    \ContinuousArgmin_k - \operatorname{Id} : \operatorname{dom} \to V_k
    \end{equation*}
    such that the following holds true:
    \begin{enumerate}
        \item
            Monotonicity: $J(\InexactArgmin_k(v)) \leq J(\ContinuousArgmin_k(v)) \leq J(v)$
            and $J(\Ccal(v)) \leq J(v)$ for all $v \in \operatorname{dom}J$.
        \item
            Continuity: $J \circ \ContinuousArgmin_k$ is continuous.
        \item
            Stability: $J(\ContinuousArgmin_k(v)) < J(v)$ if $J(v)$ is not minimal in $v + V_k$.
    \end{enumerate}
    Then the sequence $v^\nu$ produced by the TNNMG iteration converges to the unique minimizer of $J$.
\end{theorem}

Note that Theorem~\eqref{thm:gs_stationary} poses only weak conditions
on the local solvers~\eqref{eq:general_local_problem}. Roughly speaking, there needs to be a continuous
local inexact minimization operator $\ContinuousArgmin_k$ that bounds the energy decrease of the actual
minimization operator $\InexactArgmin_k$. In the important subcase that $\InexactArgmin_k$ is
exact minimization, the required properties are shown in~\cite{graeser_sander:2017}.
This allows to use rapid inexact local solvers for~\eqref{eq:general_local_problem},
when solving those local problems exactly is too expensive.  The linear correction operator $\Ccal$
(Steps~\ref{enum:linear_correction}--\ref{enum:line_search}) only needs to produce non-increasing energy
to obtain global convergence; a property that actually holds by construction.
See~\cite{graeser_sander:2017} for details.

In practice it is observed that the method degenerates to a multigrid method
after a finite number of steps, and hence multigrid convergence rates are achieved
asymptotically.

\subsection{Application to elastoplasticity problems}
\label{sec:elastoplastic_tnnmg}

Obviously, the plasticity increment functional $L$ defined in~\eqref{eq:algebraic_plasticity_functional} fits all
requirements made by the TNNMG method.  It does have a natural block structure.
Indeed, with the ordering of the coefficients given in~\eqref{eq:displacement_ordering}
and~\eqref{eq:plastic_strain_ordering}, the algebraic ansatz space $\R^N$ can be grouped in $n_1$ blocks
of size $d$ for the displacement degrees of freedom, followed by $n_2$ blocks of size $d_p +1$ for the
plastic strain and isotropic hardening values ($d_p$ if isotropic hardening is not considered).
To simplify the notation we introduce $V_i^1$, $i=1,\dots,n_1$, the $d$-dimensional space spanned
by the degrees of freedom in the $i$-th block of displacement variables, and
$V_i^2$, $i=1,\dots,n_2$, the $(d_p+1)$-dimensional space spanned
by the degrees of freedom in the $i$-th block of plastic strain and hardening variables.

In the following, $w \in \R^N$ will denote a triple $(u,p,\eta)$ of coefficients for displacements,
plastic strains, and isotropic hardening values, arranged in the ordering
\eqref{eq:displacement_ordering}--\eqref{eq:plastic_strain_ordering}.
The elastoplastic increment functional $L$ is indeed a sum as required by~\eqref{eq:min_problem}.
The first addend is
\begin{equation*}
 L_0(w)
 \colonequals
 \frac{1}{2} w^T A w - b^T w,
\end{equation*}
where $A$ and $b$ are the stiffness matrix and load vector of Chapter~\ref{sec:space_discretization}.
This functional is $C^2$ and convex.  Moreover, it is even
strictly convex and coercive if the conditions of Theorem~\ref{thm:ellipticity_of_the_bilinear_form} hold.
The remainder of $L$ is a sum over nonlinear functionals $\varphi_i$, where each $\varphi_i$ depends only
on variables from the $i$-th block.  For the $n_1$ blocks of size $d$ pertaining to the displacements,
this functional is simply zero.  For the remaining $n_2$ blocks of size $d_p+1$, it is given by
\begin{equation}
\label{eq:algebraic_dissipation_functional}
 \varphi_i(p_i,\eta_i)
 \colonequals
 \int_\Omega \theta_i(x)\,dx \cdot D\bigg( \sum_{j=1}^{d_p} B_j p_{i,j}, \eta_i\bigg),
\end{equation}
where $D$ is the dissipation function of the elastoplastic model.
By Lemma~\ref{lem:dissipation_function_properties}, any dissipation function is convex, proper, and
l.s.c.  Therefore, the $\varphi_i$ are convex, proper, and l.s.c.\ because they are the compositions
of the linear map $(p_i,\eta_i) \mapsto \big( \sum_{j=1}^{d_p} B_j p_{i,j}, \eta_i\big)$ with
the convex, proper, l.s.c.\ function $D$.  However, the $\varphi_i$ will not be differentiable,
no matter what the flow rule is.  If isotropic hardening is considered, they may even assume the
value $\infty$.

\bigskip

We now discuss the four substeps of a TNNMG iteration in the light of the linear primal
plasticity increment problem.
Step~\ref{enum:presmoothing} is one iteration of an inexact nonlinear block Gauß--Seidel method, written as a
sequence of local minimization problems~\eqref{eq:general_local_problem}. These problems
can be algorithmically challenging, but by construction, the search spaces are low-dimensional, and therefore
even minimization methods that scale badly with the problem size can be used here.
As an additional important feature, these local minimization problems typically need not be solved exactly,
which can save a lot of run-time.  We discuss several approaches for the primal elastoplastic
in Chapter~\ref{sec:smoothers}.

Step~\ref{enum:linear_correction} is an inexact Newton step on a subspace $W_\nu$ of $\R^N$,
where $W_\nu$ is constructed such that all necessary derivatives exist.
For elastoplasticity problems, the easiest approach is to take the full space $\R^N$,
and omit those plastic strain/hardening subspaces where $\varphi_i$ is not differentiable.
More technically, for a given presmoothed iterate $w^{\nu + \frac{1}{2}} = (u^{\nu + \frac{1}{2}},
p^{\nu + \frac{1}{2}}, \eta^{\nu + \frac{1}{2}})$, define the set of ``inactive'' coefficients
\begin{equation}
\label{eq:inactive_set}
 \mathcal{N}_\nu^\circ
 \colonequals
 \big\{ i = 1,\dots,n_2 \; : \; \text{$\nabla^2 \varphi_i (p_i^{\nu + \frac{1}{2}}, \eta_i^{\nu + \frac{1}{2}})$ exists and is continuous} \big\}.
\end{equation}
Then define
\begin{equation}
\label{eq:def_truncated_space}
 W_\nu
 \colonequals
 (\R^d)^{n_1}
 \oplus
 \bigoplus_{i=1}^{n_2}
 \begin{cases}
  \R^{d_p+1}  & \text{if $i \in \mathcal{N}_\nu^\circ$}, \\
  \{ 0\}^{d_p+1} & \text{otherwise}.
 \end{cases}
\end{equation}
This is the largest possible subspace in the case of the von Mises yield function
with only kinematic hardening, because then each $\varphi_i$ is $C^2$ everywhere except
for a single point.
If the Tresca dissipation function is used, then $\varphi_i$ is not differentiable
if the plastic strain has double eigenvalues,
but it is differentiable in directions that preserve the multiplicity of the eigenvalues.
In that case, a larger space than \eqref{eq:def_truncated_space} can be constructed.
However, as double eigenvalues are rare,
this discrepancy typically does not cause a noticeable difference
of the convergence speed.

The Newton correction system~\eqref{eq:inexact_linear_correction} is then
solved, but only very inexactly: the standard TNNMG method will only do one geometric or algebraic
multigrid iteration here. This is a core feature: Unlike basically all competing algorithms,
the TNNMG method does not solve a linear system in each iteration. Nevertheless it
achieves convergence rates that are competitive with predictor--corrector methods.

The projection Step~\ref{enum:projection} onto the effective domain of the functional is not strictly
necessary to guarantee convergence.
However, it is observed in practice that it extends the effectivity of the subsequent line search.
For primal plasticity with only kinematic hardening, the effective domain is the entire space, and the projection
step is void.  Otherwise, we use a cheap Euclidean projection.
The line search in Step~\ref{enum:line_search} is a one-dimensional strictly convex minimization
problem.  We solve it by using bisection to find a zero of the directional subdifferential.
Again there is no need to solve this minimization problem exactly.

\bigskip

The following global convergence result is a direct consequence of Theorem~\ref{thm:gs_stationary}.
\begin{corollary}
Suppose that the local problems~\eqref{eq:general_local_problem} are solved exactly.
Then the TNNMG iteration for the increment functional $L$ converges to the unique
minimizer of $L$ for any initial iterate $w^0 = (u^0, p^0, \eta^0) \in \R^N$.
\end{corollary}
As Theorem~\ref{thm:gs_stationary} makes actually very weak assumptions on how the local
problems are solved, global convergence also follows for a number of inexact local solvers
(see~\cite{graeser_sander:2017}) for a few examples).
The asymptotic multigrid behavior is readily observed in numerical experiments,
see Chapter~\ref{sec:numerical_experiments}. The possibility to also use inexact solvers
for~\eqref{eq:general_local_problem} is hinted at in the following section.

\subsection{Smoothers for smooth and nonsmooth yield laws}
\label{sec:smoothers}

Step~\ref{enum:presmoothing} of the TNNMG method is a nonlinear block Gauß--Seidel smoother.
It consists of a sequence of local inexact minimization problems
\begin{align*}
  \tilde{v}^k \approx \argmin L(\tilde{v}),
  \qquad
  k = 1,\dots, n_1 + n_2
\end{align*}
in the affine search spaces $\tilde{v}^{k-1} + V_k$.

In the context of primal elastoplasticity, the sequence of local minimization problems consists
of two parts:  First is a block Gauß--Seidel loop over the displacement degrees of freedom
\begin{align*}
  \tilde{w}^k \approx \argmin_{\tilde{w} \in \tilde{w}^{k-1} + V_k^1} L(\tilde{w}),
  \qquad
  k = 1,\dots, n_1
\end{align*}
in the $d$-dimensional search spaces $\tilde{w}^{k-1} + V_k^1$,
followed by a block Gauß--Seidel loop over the $(d_p+1)$-dimensional blocks containing
the plastic strain and isotropic hardening for a single element
\begin{align*}
  \tilde{w}^k \approx \argmin_{\tilde{w} \in \tilde{w}^{k-1} + V_k^2} L(\tilde{w}),
  \qquad
  k = 1,\dots, n_2.
\end{align*}

Let first $k$ be a displacement block.  Then minimizing $L$ in the $d$-dimensional space
$V_k^1$ is equivalent to minimizing
\begin{equation*}
 L_k : \R^d \to \R,
 \qquad
 L_k(u_k) \colonequals \frac{1}{2} u_k^T E_{kk} u_k - u_k^T r_k + \text{const}.
\end{equation*}
The matrix $E_{kk}$ is of the form given in~\eqref{eq:elasticity_matrix_entries},
and $r_k \in \R^d$ is the residual
\begin{equation*}
 r_k
 =
 b_k
 -
 \sum_{\substack{i=1\\ i\neq k}}^{n_1} (E_{ki})u_i
 -
 \sum_{i=1}^{n_2} (C_{ki})\underbrace{\begin{pmatrix}p_i \\ \eta_i \end{pmatrix}}_{\in \R^{d_p} \times \R}.
\end{equation*}
This is a quadratic minimization problem in $d$ variables with a symmetric positive definite matrix.
For efficient overall convergence rates it is sufficient to solve it inexactly using one sweep
of a scalar Gauß--Seidel method, but one may as well solve it directly, using, e.g.,
Cholesky decomposition.

If $k$ is one of the plastic strain/isotropic hardening blocks, then the local functional has the form
\begin{align}
 \nonumber
 L_k & \; : \R^{d_p + 1} \to \R \cup \{ \infty \}, \\
\label{eq:local_plastic_strain_problem}
 L_k\begin{pmatrix} p_k \\ \eta_k \end{pmatrix}
 & =
 \frac{1}{2} \begin{pmatrix} p_k \\ \eta_k\end{pmatrix}^T P_{kk} \begin{pmatrix}p_k \\ \eta_k \end{pmatrix}
  - \begin{pmatrix}p_k \\ \eta_k\end{pmatrix}^T r_k
   + \lambda_k D\Big( \sum_{j=1}^{d_p} B_j p_{k,j}, \eta_k\Big),
\end{align}
with the constant $\lambda_k \colonequals \int_\Omega \theta_k\,dx$ depending only on the grid.
The matrix $P_{kk}$ is given by~\eqref{eq:plastic_strain_matrix_entries}, and the residual $r_k$ has the form
\begin{equation*}
 r_k
 =
 b_k - \sum_{i=1}^{n_1}(C^T)_{ki} u_i.
\end{equation*}
Under the assumptions of Theorem~\ref{thm:ellipticity_of_the_bilinear_form}, $L_k$ as defined
in~\eqref{eq:local_plastic_strain_problem} is a strictly convex,
coercive minimization functional.

Depending on the actual dissipation function $D$, minimizing this functional can nevertheless be a challenge.
However, each local problem depends on only $d_p+1$ variables (or $d_p$
if there is no isotropic hardening), a quantity that is small, and independent of the mesh size.
This allows to employ algorithms from convex optimization even if they scale badly with the
problem size.  Also, convergence of the TNNMG method does not require to solve
the local problems exactly.  We discuss a few examples in turn.

\subsubsection{Example: The von Mises dissipation function}

If $D$ is the von Mises dissipation function, and isotropic hardening is omitted, then the local functional
$L_k$ has the form
\begin{equation}
\label{eq:von_mises_local_functional}
 L_k : \R^{d_p} \to \R,
 \qquad
 L_k(p)
 \colonequals
 \frac{1}{2} p^T P_{kk} p - p^T r_k + \lambda_k \sigma_c \norm{B(p)}.
\end{equation}
It is well known that the minimizer of this can be computed in closed form.
The following result is proved in~\cite[Prop.\,7.1]{alberty_carstensen_zarrabi:1999}.
The matrix $\bC$ is the Hooke tensor of the St.\,Venant--Kirchhoff material~\eqref{eq:stvenant_kirchhoff},
$\mu$ is the second Lamé parameter, and $k_1$ is the kinematic hardening modulus.
\begin{lemma}
Given $R \in \mathbb{S}^d$ and $\sigma_c > 0$.  Then
\begin{equation*}
P^*
\colonequals
\frac{\max\{(\norm{R^D} - \sigma_c ),0\}}{2\mu + k_1}
\frac{R^D}{\norm{R^D}}
\end{equation*}
is the unique minimizer of
\begin{equation*}
 \mathcal{L}(P)
 \colonequals
 \frac{1}{2} (\bC + k_1)P : P - P : R + \sigma_c \norm{P}
\end{equation*}
among the trace-free symmetric $d \times d$-matrices.
\end{lemma}

Alternatively, inexact solvers can be used, for example if the elastic law is not
the St.\,Venant--Kirchhoff one.  For this, note that $L_k$ as defined in~\eqref{eq:von_mises_local_functional}
is continuously differentiable for all $p \neq 0$, with gradient given by
\begin{align*}
 (\nabla L_k(p))_l
 & =
 P_{kk}p - r_k
 + \lambda_k \sigma_c \sum_{i,j=1}^d \frac{B(p)_{ij}}{\norm{B(p)}} \cdot \frac{\partial B(p)_{ij}}{\partial p_l},
 \qquad
 l = 1,\dots,d_p.
\end{align*}
At $p=0$, the functional is not differentiable, but the forward directional derivative exists.
For a given direction $v \in \R^{d_p}$, we have
\begin{align*}
\frac{dL_k(0)}{dv}
 \colonequals
 \lim_{t \searrow 0} \frac{1}{t} \big[ L_k (0 + tv) - L_k(0)\big]
 =
 - v^T r_k + \lambda_k \sigma_c \lim_{t \searrow 0} \frac{1}{t} \norm{B(tv)}.
\end{align*}
The last term is equal to $\lambda_k \sigma_c \norm{B(v)} = \lambda_k \sigma_c\norm{v}$, because
$p \mapsto B(p)$ is an isometry by construction.
With this special symmetry, the steepest descent direction at $p=0$ can easily be computed
\begin{equation*}
 d_\text{min}
 =
 \argmin_{\norm{v}=1} \frac{dL_k(0)}{dv}
 =
 \argmin_{\norm{v}=1} \Big[- v^T r_k + \lambda_k \sigma_c \norm{v} \Big]
 =
 \frac{r_k}{\norm{r_k}}.
\end{equation*}

If there is isotropic hardening, the local functional takes the slightly more general form
\begin{equation*}
 L_k : \R^{d_p} \times \R \to \R \cup \{ \infty \}
 \qquad
 L_k\begin{pmatrix}p \\ \eta \end{pmatrix}
 \colonequals
 \frac{1}{2} p^T P_{kk} p - \begin{pmatrix} p\\ \eta\end{pmatrix}^T r_k + \frac{1}{2} k_2 \eta^2 + \lambda_k \sigma_c \norm{B(p)},
\end{equation*}
but is now to be minimized under the constraint
\begin{equation*}
 \sigma_c \norm{B(p)} = \sigma_c \norm{p} \le \eta.
\end{equation*}
One possible approach for this is an inexact projected steepest descent method, generalizing the
formulas for the steepest descent direction of the functional without isotropic hardening.
The admissible set is a circular cone, onto which a projection can be computed easily.

\subsubsection{Example: The Tresca dissipation function}

If the Tresca dissipation functional is used, and isotropic hardening is omitted, then the local functional
has the form
\begin{equation*}
 L_k(p)
 \colonequals
 \frac{1}{2} p^T P_{kk} p - p^T r_k + \lambda_k \sigma_c \rho(B(p)),
\end{equation*}
where $\rho(B(p))$ is the spectral radius of the matrix $B(p)$.
This functional is nonsmooth whenever the spectral radius is realized
by more than one eigenvalue.  It is not clear how to compute a steepest-descent direction for this
functional; however, the computation of subdifferentials is relatively straightforward.

As in the von Mises case, isotropic materials are particularly simple.
\citet{rencontre_bird_martin:1992} show how the principal
directions of the minimizer of $L_k$ can be computed a priori. This reduces the minimization
problem to a problem in $\R^d$, where the nonsmoothness is the infinity norm
$\abs{x}_\infty :  x \mapsto \max_{1,\dots,d} \abs{x_i}$.
To compute the subdifferential of $\abs{\cdot}_\infty : x \mapsto \max_{i=1,\dots,d} \abs{x_i}$, we rewrite it as
\begin{equation*}
 \abs{x}_\infty = \max_{i=1,\dots,d} \{x_i, -x_i \},
\end{equation*}
and note that this is a subsmooth function in the sense of \cite[Chap.\,10]{rockafellar_wets:2010}.
Therefore, we can use Theorem~10.31 of \cite{rockafellar_wets:2010} to obtain that
\begin{align*}
 \partial \abs{x}_\infty
 =
 \operatorname{conv} \big\{ & \phantom{-} \text{$e_i$ for all $i$ such that $\phantom{-}x_i = \abs{x}_\infty$}, \\
         & \text{$-e_i$ for all $i$ such that $-x_i = \abs{x}_\infty$} \big \},
\end{align*}
where $e_i = \nabla x_i$ is the $i$-th canonical basis vector of $\R^d$.

For the general case, we observe that the spectral radius function $\rho : \mathbb{S}^d \to \R$
can be written as a composition
\begin{equation*}
 \rho(X) = (\abs{\cdot}_\infty \circ \lambda)(X),
\end{equation*}
where $\lambda$ maps the symmetric matrix $X$ to its eigenvalues (in some order), and $\abs{\cdot}_\infty$
is symmetric, i.e., invariant under permutations of its arguments~\cite{drusvyatskiy_paquette:2015,lewis:1996}.
Using a general result from matrix analysis, we can reduce the computation of the subdifferential
of $\rho$ to the computation of the subdifferential of $\abs{\cdot}_\infty$.

\begin{theorem}[\cite{drusvyatskiy_paquette:2015,lewis:1996}]
 Suppose that the function $f : \R^d \to \R \cup \{ \infty \}$ is convex and symmetric.
 Then $f \circ \lambda : \mathbb{S}^d \to \R$ is convex, and
\begin{equation*}
 \partial(f \circ \lambda)(X)
 =
 \Big \{ U (\operatorname{diag} v) U^T
  \; : \;
  v \in \partial f(\lambda(X)), \;U \in O_X \Big\},
\end{equation*}
where $O_X$ is the set of all orthogonal matrices $U$ for which
\begin{equation*}
 X = U (\operatorname{diag}\lambda(X)) U^T.
\end{equation*}
\end{theorem}
Hence, to compute an element of the subdifferential of $\rho = \abs{\cdot}_\infty \circ \lambda$
at a matrix $X$, we first compute an eigendecomposition $U (\operatorname{diag} \lambda(X)) U^T$ of $X$.
Then, for any subgradient $v$ of $\abs{\cdot}_\infty$ at $\lambda(X)$, the matrix $U (\operatorname{diag} v) U^T$
will be in the subdifferential of $\rho$.
With this result we can use methods from the families of subgradient methods
or bundle methods~\cite{geiger_kanzow:2002} as local solvers of the smoother.  Such methods are known
to converge only very slowly.  This is, however, not a problem, because smoothers for the
TNNMG method only need to produce a certain quantity of energy decrease, but no exact solution.

From results in \cite{drusvyatskiy_paquette:2015} it also follows how to compute second derivatives
of the Tresca dissipation wherever these second derivatives exist.  This is precisely what is
needed to set up the truncated linear correction problem~\eqref{eq:inexact_linear_correction}.

\subsection{Constructing positive definite correction matrices}

The inexact linear correction (Step~\ref{enum:linear_correction}) involves a linear system on an iteration-dependent subspace $W_\nu$
of the total ansatz space $\R^N$.  The subspace is constructed such that the matrix of second derivatives of $L$ is
well-defined there, and by strong convexity it is positive definite.  In Chapter~\ref{sec:elastoplastic_tnnmg}
we explained that a simple way to construct the subspaces $W_\nu$ is by simply omitting those blocks of coefficients
where the nonlinearity functions $\varphi_i$ are not $C^2$.  This approach is called
\emph{truncation}.

Usually, the best way to implement truncation is to keep working with a Newton matrix of size $N \times N$,
and modify the entries for the truncated degrees of freedom.
The naive way to do this is by overwriting the corresponding
rows and columns of the matrix and right-hand-side vector with zeros, and placing `1's at the relevant
diagonal matrix entries. This ensures that the corrections for the truncated
degrees of freedom are zero, and yet the overall matrix remains positive definite.
However, the value~`1' is arbitrary, and its choice can influence the condition number of the modified
Newton matrix.  This is undesirable, because the error reduction rate of the multigrid iteration step
used to approximate the solution of the linear correction problem depends on the
matrix condition number.  The standard approach of the TNNMG method is therefore to
keep the zeros on the diagonal, and to cope with the resulting semidefiniteness.  This is
quite straightforward: the
multigrid smoother needs to be able to handle zeros on the diagonal, and the prolongation operators
have to be modified to keep coarse grid contributions from spilling into truncated degrees of
freedom~\cite[Chap.\,4.3.1]{graeser:2011},\cite{graeser_sander:2017}.

However, for linear primal plasticity problems we can do better. Since we are using
a piecewise constant discretization of
the plastic strain, a trick allows to use a positive definite matrix instead of a semidefinite
one.  As a consequence, a standard linear multigrid step can be used to compute the linear corrections.

For a given iteration $\nu$, call $H_\nu$ the truncated Newton matrix of $L$ at $w^{\nu+\frac12}$.
Following the ordering of the coefficients introduced in Chapter~\ref{sec:space_discretization},
it has the block form
\begin{equation}
\label{eq:truncated_hesse_matrix}
 H_\nu
 =
 \begin{pmatrix}
  E & C T_\nu \\
  T_\nu^T C^T & T_\nu^T \widetilde{P}_\nu T_\nu
 \end{pmatrix},
\end{equation}
where $E$ is the matrix~\eqref{eq:elasticity_matrix_entries} that relates displacements to displacements,
and $C$ contains the coupling terms~\eqref{eq:coupling_matrix_entries} between displacements and plastic strain.
The matrix $\widetilde{P}_\nu$ is the sum
\begin{equation*}
 \widetilde{P}_\nu
 =
 P +
 \operatorname{diag} \big[ \nabla^2 \varphi_i(R_i^2 (w^{\nu + \frac{1}{2}})), \; i \in \mathcal{N}^\circ_\nu\big],
\end{equation*}
with $P$ given by~\eqref{eq:plastic_strain_matrix_entries}.
The truncation matrix $T_\nu$ is quadratic with $(d_p+1)n_2$ rows, and defined by
\begin{equation*}
 (T_\nu)_{ij}
 \colonequals
 \begin{cases}
  1 \cdot I_{d_p+1} & \text{if $i=j$ and $i \in \mathcal{N}_\nu^\circ$} \\
  0 \cdot I_{d_p+1} & \text{otherwise},
 \end{cases}
\end{equation*}
($I$ the identity matrix).
Multiplication with $T_\nu$ from the left and right zeros out matrix rows and columns, respectively.

For the rest of this section we omit the iteration index $\nu$ for simplicity.
The trick is now that, due to the zero-order discretization
of the plastic strain, $\widetilde{P}$ is block diagonal, and so is $T^T \widetilde{P} T$.
Therefore, its Moore--Penrose pseudo inverse $(T^T \widetilde{P} T)^+$ can be
computed cheaply.  We use this to eliminate the plastic strain and isotropic hardening values
from the correction equation, and apply the multigrid iteration to the Schur complement system.
The Schur complement matrix is positive definite even though the complete Newton matrix is not.
\begin{lemma}
 \label{lem:schur_complement_positive_definite}
 The Schur complement
 $S \colonequals E - C (T^T\widetilde{P}T)^+ C^T$ is positive definite.
\end{lemma}
\begin{proof}
Write $H$ as
 \begin{equation*}
  H
  =
  \begin{pmatrix}
   E & \widehat{C} \\ \widehat{C}^T & \widehat{P}
  \end{pmatrix},
 \end{equation*}
and observe that it is positive semidefinite.  Further, $E$ is positive definite, and by
construction we have
\begin{equation}
\label{eq:kernel_of_truncated_hessian}
 \ker H = \{ 0 \} \times \ker \widehat{P}.
\end{equation}
Let $u \neq 0$, and set $p \colonequals - \widehat{P}^+\widehat{C}^T u$.  By \eqref{eq:kernel_of_truncated_hessian},
$(u,p) \notin \ker H$, and therefore
\begin{align*}
 0 & <
   \begin{pmatrix}
    u \\ -\widehat{P}^+\widehat{C}^T u
   \end{pmatrix}^T
   \begin{pmatrix}
    E & \widehat{C} \\
    \widehat{C}^T & \widehat{P}
   \end{pmatrix}
   \begin{pmatrix}
    u \\ -\widehat{P}^+\widehat{C}^T u
   \end{pmatrix}.
\end{align*}
On the other hand, the right hand side of this is equal to
\begin{equation*}
u^T (E - \widehat{C}\widehat{P}^+\widehat{C}^T) u
 =
u^T (E - CT (T^T\widetilde{P}T)^+T^TC^T) u
 =
u^T S u,
\end{equation*}
and the assertion is shown.
\end{proof}

Using this we can show that solving the Schur complement system does yield
a solution to the original semidefinite problem.  We write $(g_1, T^T g_2)$ for the negative
truncated gradient of the increment functional $L$.
\begin{lemma}
 A vector $(u,p)$ is a solution of the truncated system
\begin{equation*}
 \begin{pmatrix}
  E & CT \\
  T^TC^T & T^T\widetilde{P}T
 \end{pmatrix}
 \begin{pmatrix}
  u \\ p
 \end{pmatrix}
 =
 \begin{pmatrix}
  g_1 \\ T^T g_2
 \end{pmatrix}
\end{equation*}
 if $u$ is a solution of the Schur complement equation
 \begin{equation*}
  S u \colonequals (E - C (T^T\widetilde{P}T)^+ C^T) u = g_1 - CT^T\widetilde{P}^+ T^T g_2,
 \end{equation*}
 and $p$ is a solution of
 \begin{equation}
 \label{eq:second_schur_equation}
  (T^T \widetilde{P} T)p
  =
  T^T g_2 - T^T C^T u.
 \end{equation}
\end{lemma}

\begin{proof}
 By Lemma~\ref{lem:schur_complement_positive_definite}, $S$ is invertible, and therefore
 $u = S^{-1}(g_1 - CT\widetilde{P}^+ Tg_2)$ is well-defined.  The right-hand side of~\eqref{eq:second_schur_equation}
 is obviously in the range of the matrix $T^T \widetilde{P} T$, and therefore \eqref{eq:second_schur_equation}
 has a solution $p$.  This solution is unique up to the elements of $\ker T^T \widetilde{P} T = \ker T$.
 Direct calculations then show that $(u,p)$ solves the original system.
\end{proof}

\subsection{Comparison with a predictor--corrector method}
\label{sec:predictor_corrector_methods}

The TNNMG method is closely related to the predictor--corrector method traditionally used to solve
primal plasticity problems~\cite{han_reddy:2013}.  Of the different variations that have been proposed, the closest
one is the predictor--corrector method with a consistent tangent predictor and a line search.
This method was studied, e.g., in~\cite{caddemi_martin:1991,martin_caddemi:1994} and~\cite{han_reddy:2013},
and is generally acknowledged to be the most efficient predictor--corrector algorithm for the
given problem.  We will simply refer to it as ``the predictor--corrector method'' in the following.

Let $\nu \in \mathbb{N}_0$ be the iteration number, and $w^\nu = (u^\nu, p^\nu, \eta^\nu)$ the current iterate.
In our notation, and following the presentation of~\cite{han_reddy:2013}, the method has the following form.
\begin{enumerate}
 \item
  \textbf{Consistent tangent predictor}
  \begin{enumerate}
   \item Define the truncated dissipation functional%
   \footnote{Here the method assumes that $\varphi$ is $C^2$-differentiable for all $p \neq 0$.
     Generalizations appear straightforward, but seem to be absent from the literature.
     }
    \begin{equation*}
     \tilde{\varphi}_i^\nu(p)
     \colonequals
     \begin{cases}
      \varphi_i(p) & \text{if $p_i^\nu = 0$} \\
      0 & \text{otherwise},
     \end{cases}
     \qquad
     \forall p \in \R^{d_p}
    \end{equation*}
    where the $\varphi_i$, $i=1,\dots,n_2$ are the algebraic dissipation functionals
    defined in~\eqref{eq:algebraic_dissipation_functional}.
   \item Define the semi-truncated Hesse matrix at $w^\nu$ as
     \begin{equation*}
       H_\nu^\text{pc}
        =
       \begin{pmatrix}
         E   & C \\
         C^T & P+ \operatorname{diag}[\nabla^2 \tilde{\varphi}_i(p_i^\nu,\eta_i^\nu)]
       \end{pmatrix},
     \end{equation*}
     and the semi-truncated gradient
     \begin{equation*}
      G_\nu
      \colonequals
      (w^\nu)^T A - b^T
      +
      \sum_{i=1}^{n_2} \nabla \tilde{\varphi}^\nu_i(p_i^\nu, \eta_i^\nu).
     \end{equation*}
     The latter is simply the gradient of \eqref{eq:algebraic_plasticity_functional},
     with $\varphi_i$ replaced by $\tilde{\varphi}_i^\nu$.%
     \footnote{
     Note that $H_\nu^\text{pc}$ differs from $H_\nu$ as defined in \eqref{eq:truncated_hesse_matrix}.
     While $H_\nu$ has zero rows and columns for degrees of freedom where the dissipation is not
     differentiable, $H_\nu^\text{pc}$ keeps the elastic part there.
     However, as the corresponding degrees of freedom are held fixed there is no practical difference.
     }
   \item Then solve the linear system
    \begin{equation}
     \label{eq:predictor_problem}
     H_\nu^\text{pc} c^\nu = - G_\nu.
    \end{equation}
    for the Newton correction $c^\nu$, but constrain the correction to zero for all blocks $i=1,\dots,n_2$
    with $p_i^\nu = 0$.
  \end{enumerate}
  \item
   \textbf{Damped update / line search}
   \begin{enumerate}
     \item
        Compute $\rho_\nu \in [0,\infty)$ such that
        $L(w^\nu + \rho_\nu c^\nu) < L(w^\nu)$.
     \item
        Set $w^{\nu+\frac12} = w^\nu + \rho_\nu c^\nu$
    \end{enumerate}
    \item
        \textbf{Corrector}
        \begin{enumerate}
            \item
                Set $\tilde{w}^0 = w^{\nu + \frac 12}$
            \item
                For $k=1,\dots,n_2$ do
                \begin{enumerate}
                    \item
                        Compute $\tilde{w}^k \in \tilde{w}^{k-1} + V_k$ as
                        \begin{align}
                        \label{eq:corrector_local_problems}
                            (\tilde{p}^k, \tilde{\eta}^k) = \argmin_{w \in \tilde{w}^{k-1} + V_k} L(w),
                            \qquad
                            \tilde{u}^k = \tilde{u}^{k-1}
                        \end{align}
                \end{enumerate}
            \item
                Set $w^{\nu+1} = \tilde{w}^{n_2}$
        \end{enumerate}
        Here, $V_k$ is the space of all internal variables at grid element $k$.
\end{enumerate}
To the best of our knowledge, no convergence proof for this method appears in the literature.
In~\cite{caddemi_martin:1991,martin_caddemi:1994} and \cite{han_reddy:2013}, energy decrease of the method is shown,
which, however, by itself does not imply convergence.  Furthermore, the proofs require that both
the predictor problem~\eqref{eq:predictor_problem} and the corrector problems~\eqref{eq:corrector_local_problems}
be solved exactly.

Comparing this method with the algorithm from Chapter~\ref{sec:tnnmg_method} it is obvious that the two
are closely related. Indeed, the nonlinear presmoother of TNNMG corresponds to the corrector,
and the linear correction step of TNNMG corresponds to the tangent predictor.
Both methods use a line search.
TNNMG includes a projection onto the admissible set, but in situations without isotropic hardening
this projection is void anyway. While not required by the theory we have found that
it speeds up the computations, because it tends to permit longer step lengths $\rho_\nu$.

The presmoother of TNNMG loops over all degrees of freedom, whereas the corrector touches only the
plastic strains and internal variables.  While this difference has hardly any practical influence,
there are important consequences for the convergence theory: if the TNNMG presmoother stalls,
then we must be at the global minimizer. Indeed, looking at the convergence proof
in~\cite{graeser_sander:2017}, one sees that global convergence of the presmoother alone
is what guarantees convergence of the overall method.

The decisive difference between the two methods is that the TNNMG method
allows to solve both the predictor problem and the local corrector problems very inexactly.
This leads to the observed drastic difference in run-time.
Indeed, at each iteration, the tangent predictor solves an entire elastic problem, whereas the
TNNMG method does only a single iteration of a geometric or algebraic multigrid  method
for the same problem.  This is of course much cheaper.  Nevertheless,
guaranteed global convergence is retained.

Similarly, the TNNMG method converges even if the presmoother minimization problems are solved
only inexactly.  This is not clear for the corrector steps of a predictor--corrector method,
and the question has apparently never been asked in the literature.  This allows to use
approximate minimization schemes for the local subproblems like the ones mentioned in
Section~\ref{sec:smoothers}, when full minimization is expensive.

While the algorithmic differences may appear small, their effect is dramatic.
As the TNNMG method does only one multigrid iteration for the linear correction / predictor problems
instead of a full solve, it needs about two to three times as many iterations to reach a given
accuracy.  However, each iteration needs much less wall-time than a predictor--correction step,
which more than makes up for the increased iteration numbers. The outcome is a clear win
for TNNMG, with more and more lead as the problems get bigger.

\subsection{Extension to gradient plasticity}
\label{sec:gradient_plasticity}

Gradient plasticity is a classic extension of the plasticity model considered in this paper. Unlike in the basic model,
which considers only point values of the plastic strain $\mathbf{p}$ and the hardening variable $\eta$,
strain gradient models introduce additional terms involving the gradients of $\mathbf{p}$ and/or the hardening variables.
For example, in the Aifantis model~\cite{aifantis:1984,han_reddy:2013}, the weak form of the problem is formally
given by~\eqref{eq:variational_inequality}, but the bilinear form contains an additional term involving the
gradient of $\eta$
\begin{equation*}
 a(\mathbf{w},\tilde{\mathbf{w}})
 \colonequals \int_\Omega \Big[\bs{C}(\bs{\varepsilon}(\bs{u}) - \bs{p}) : (\bs{\varepsilon}(\tilde{\bs{u}}) - \tilde{\bs{p}})
   + k_2 \eta \tilde{\eta} + k_3 \nabla \eta \cdot \nabla \tilde{\eta} \Big]\,dx.
\end{equation*}
Consequently, in this model, $L^2$ is not a suitable space for $\eta$ anymore,
and a first-order Sobolev spaces has to be used instead.

The TNNMG method is well suited to solve the increment problem of gradient plasticity models.
In the following we describe the changes for a model containing the gradient of $\eta$ exclusively. The same techniques
also work if the gradient of $\bs{p}$ is involved. The crucial insight
is that, even though gradients of $\eta$ appear, the nonsmooth dissipation function remains
unchanged, and still depends on point values (in a suitable weak sense) of $\mathbf{p}$ and $\eta$ exclusively.
Therefore, the functional $L$ still decouples additively into a quadratic part and a point-wise non-smooth functional.

Getting this decoupling into the algebraic setting requires one trick.  As $\eta$ is now an $H^1$-function,
first-order finite elements are required for a proper space discretization.  To still be able to group the degrees of
freedom for $\bs{p}$ and $\eta$ together, we also use first-order finite elements for $\bs{p}$.
Hence, the algebraic increment functional corresponding to \eqref{eq:algebraic_plasticity_functional_unlumped} is
\begin{align}
 \nonumber
 \widetilde{L} & \; : (\R^d)^{n_1} \times (\R^{d_p} \times \R)^{n_1} \to \R \cup \{ \infty \} \\
 \label{eq:unlumped_gradient_plasticity_functional}
 \widetilde{L}(w)
 & \colonequals
 \frac{1}{2} w^T A w - b^T w
   + \int_\Omega D\bigg(\sum_{i=1}^{n_1} \phi_i(x) \sum_{j=1}^{d_p} B_j p_{i,j}, \sum_{i=1}^{n_1} \phi_i(x) \eta_i\bigg)\,dx.
\end{align}
However, as the first-order nodal basis functions $\phi_i$ are not piecewise constant, the sum over $i$ cannot
be pulled out of the dissipation function $D$, and the algebraic functional does {\em not} decouple additively!
We therefore {\em approximate} $\widetilde L$ by the functional
\begin{equation*}
 L(w)
  \colonequals
 \frac{1}{2} w^T A w - b^T w
    + \sum_{i=1}^{n_1} \int_\Omega \phi_i\,dx \cdot D\bigg(\sum_{j=1}^{d_p} B_j p_{i,j}, \eta_i \bigg).
\end{equation*}
The resulting functional is still coercive, convex, and lower semi-continuous, and even strictly convex if $\widetilde{L}$ is.
It can be interpreted as an approximation of the integral
in~\eqref{eq:unlumped_gradient_plasticity_functional} by a quadrature rule using only the Lagrange nodes of $\{\phi_i\}_{i=1}^n$.
This trick has been used frequently in other applications of the TNNMG
method~\cite{graeser_sack_sander:2009,graeser_sander:2014,graeser:2011}.

\section{Numerical experiments}
\label{sec:numerical_experiments}

In this chapter we present numerical results that demonstrate the efficiency of the TNNMG solver.
We use a benchmark problem involving von Mises plasticity and kinematic hardening, and compare
with a predictor--corrector method with a consistent tangent predictor.

\subsection{Square with a hole}

Our benchmark example uses the von Mises yield condition and kinematic hardening on a two-dimensional domain.
It is inspired by a similar benchmark proposed in~\cite{stein_wriggers_rieger_schmidt:2002}.

\begin{figure}
 \begin{center}
  \begin{minipage}{0.3\textwidth}
  \includegraphics[width=\textwidth]{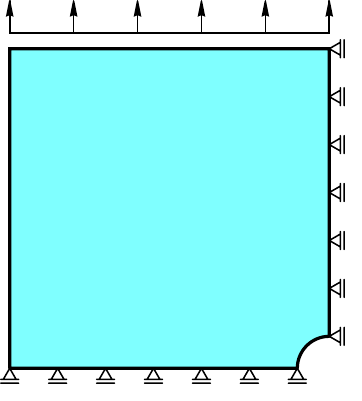}
  \end{minipage}
  \hspace{0.05\textwidth}
  \begin{minipage}{0.3\textwidth}
  \includegraphics[width=\textwidth]{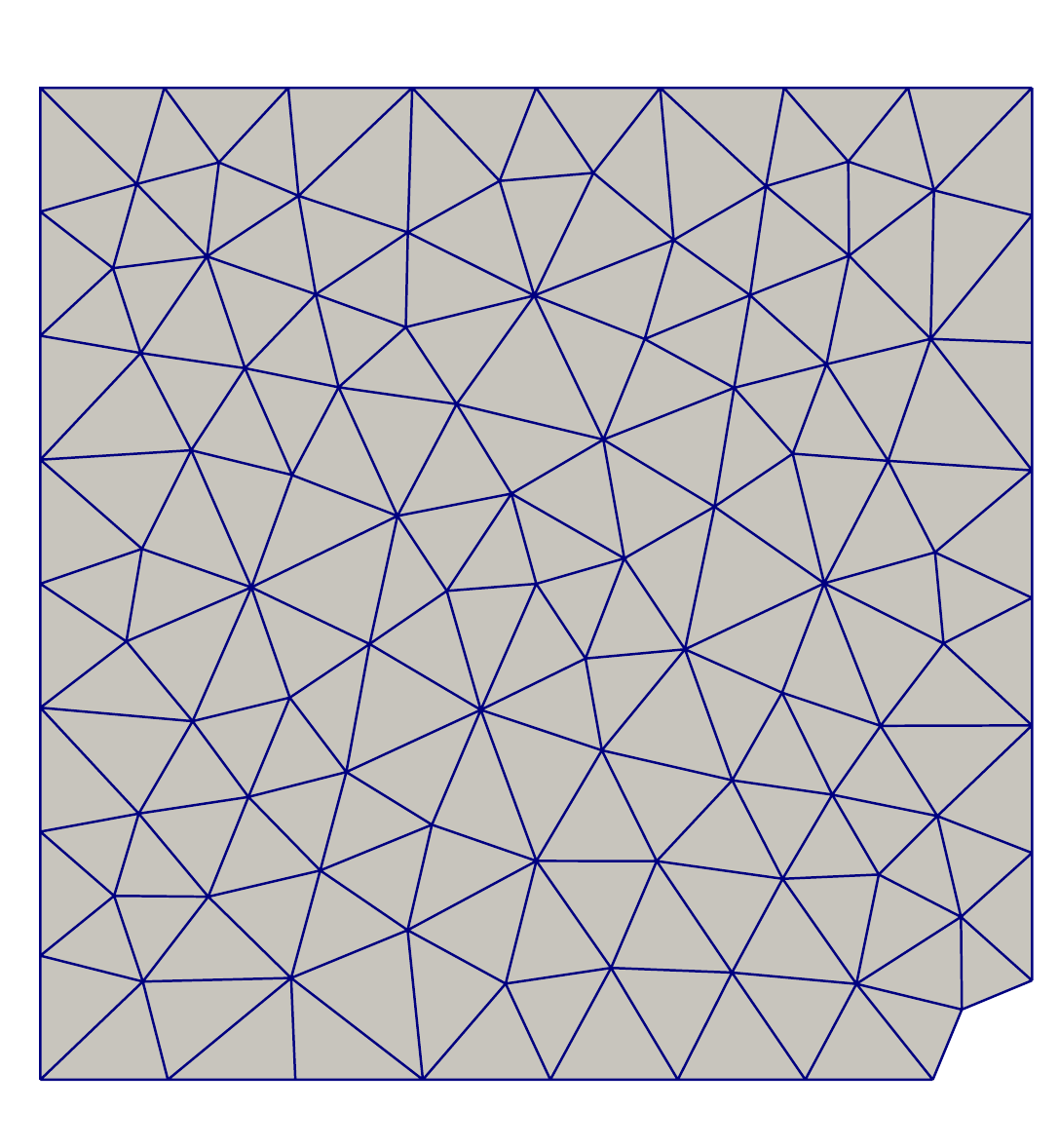}
  \end{minipage}
  \begin{minipage}{0.3\textwidth}
  {\small
   \begin{tabular}{r|r|r}
 levels & cells & vertices \\
 \hline
   1    & 176    &  105 \\
   2    & 704    & 385 \\
   3    & 2\,816   & 1\,473 \\
   4    & 11\,264  & 5\,761 \\
   5    & 45\,056  & 22\,785 \\
   6    & 180\,224 & 90\,625 \\
 \end{tabular}
 }
 \end{minipage}
 \end{center}
 \caption{Geometry and boundary conditions of the numerical benchmark.
     As the grid is refined, the domain boundary approximates the hole better and better.}
 \label{fig:plate_with_hole_geometry}
\end{figure}

Geometry and boundary conditions are illustrated in Figure~\ref{fig:plate_with_hole_geometry}.  We consider a $10 \times 10$ square
domain, from which the quarter of a unit circle has been cut out from the lower right corner (lengths are measured
in millimeters).  Normal displacements of the lower
and right sides of the square are prohibited by setting the Dirichlet boundary conditions
\begin{equation*}
  u_1(10, x_2 ) = 0, \quad x_2 \in (1, 10),
  \qquad
  u_2 (x_1, 0) = 0, \quad x_1 \in (0, 9).
\end{equation*}
Otherwise, the square is free to move. A time-dependent upward-pointing surface force
\begin{equation*}
  \langle l(t), \tilde{\bu}\rangle = 100\,t \int_0^{10} \tilde{\bu}(x_1 , 10)\cdot \be_2 \, dx_1.
\end{equation*}
is applied to the upper horizontal side.
For the material parameters we choose Lam\'e parameters  $\mu = 6.5\cdot 10^6\,[\mathrm{N}/\mathrm{mm}]$
and $\lambda = 10^7\,[\mathrm{N}/\mathrm{mm}]$.
The yield stress is set to $\sigma_c = 450\,[\mathrm{N}/\mathrm{mm}]$.
We use a kinematic hardening modulus of $k_1 = 3\cdot 10^6\,\mathrm{N/mm}$.
For the implementation of the inactive set $\mathcal{N}_\nu^\circ$ in~\eqref{eq:inactive_set}
we say that $\nabla^2\varphi_i$ exists and is continuous for all arguments $p$ with $\norm{p} \ge 10^{-10}$.

The domain is discretized using an unstructured triangle grid consisting of 176 triangular elements, pictured in Figure~\ref{fig:plate_with_hole_geometry}, right.
All grids used later for convergence rate measurements are created from this grid by uniform refinement.
The table in Figure~\ref{fig:plate_with_hole_geometry} shows the problem sizes for the different refinement levels.

\begin{figure}
 \begin{center}
  \includegraphics[width=0.2\textwidth]{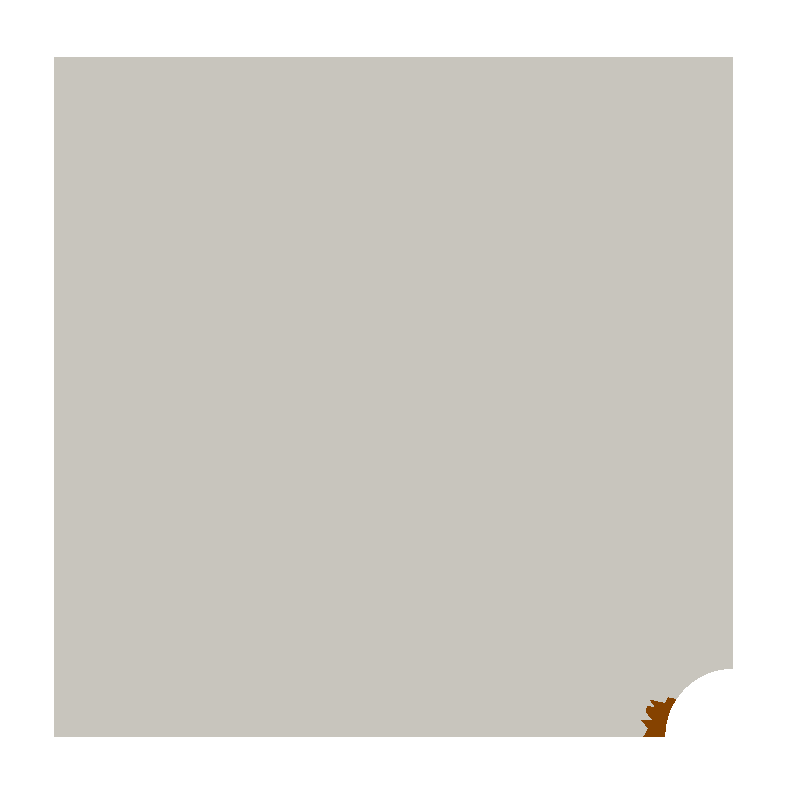}
  \includegraphics[width=0.2\textwidth]{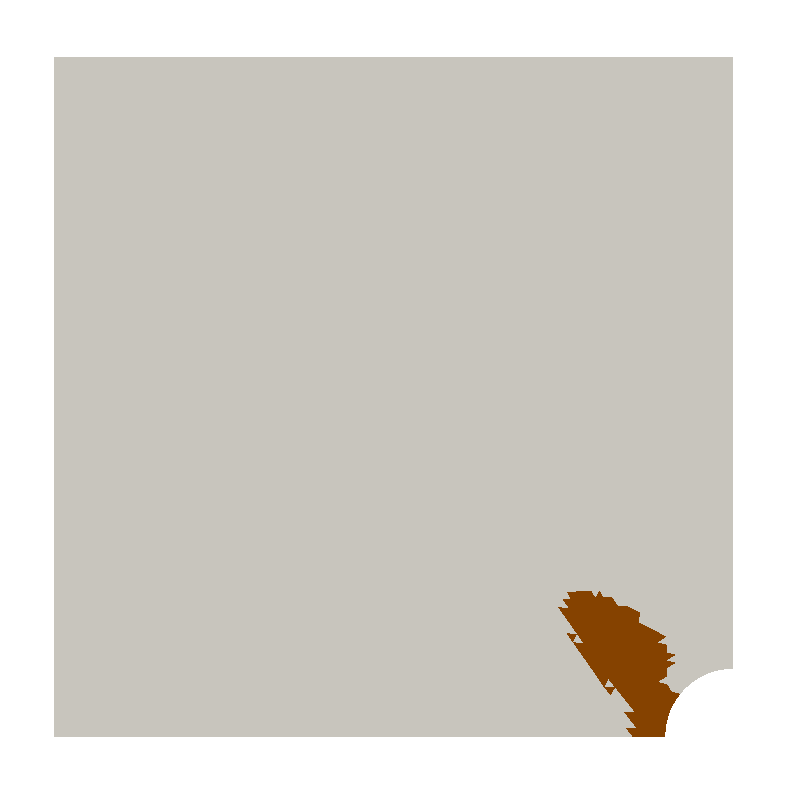}
  \includegraphics[width=0.2\textwidth]{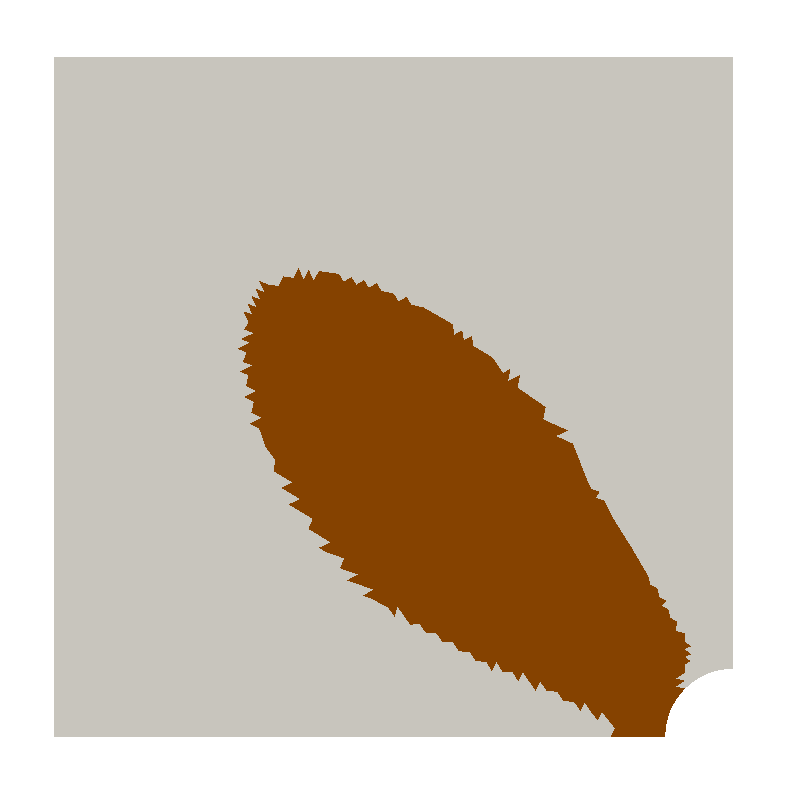}
  \includegraphics[width=0.2\textwidth]{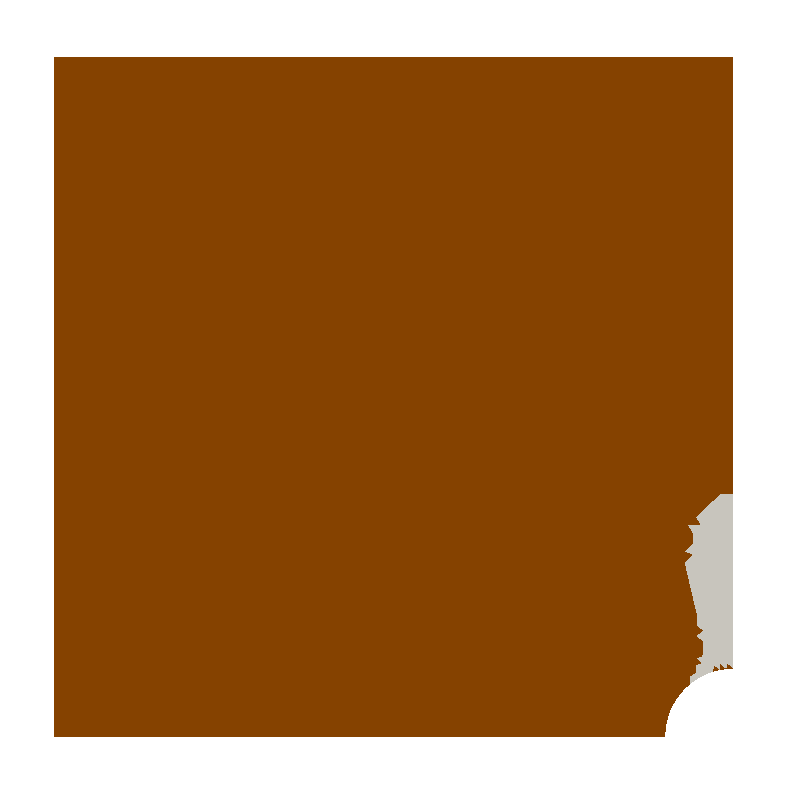}
 \end{center}
 \caption{Zone with plastic deformation at time steps 3, 4, 5, and 6}
 \label{fig:plastic_zone_evolution}
\end{figure}

We run the simulation for 20 time steps.  Figure~\ref{fig:plastic_zone_evolution} shows the evolution of the area of plastic deformation
using a grid obtained by three steps of uniform refinement.  While the object remains completely elastic for the first time step,
large parts of it show plastic flow rather quickly.

\bigskip

\begin{figure}
 \begin{center}
  \includegraphics[width=0.47\textwidth]{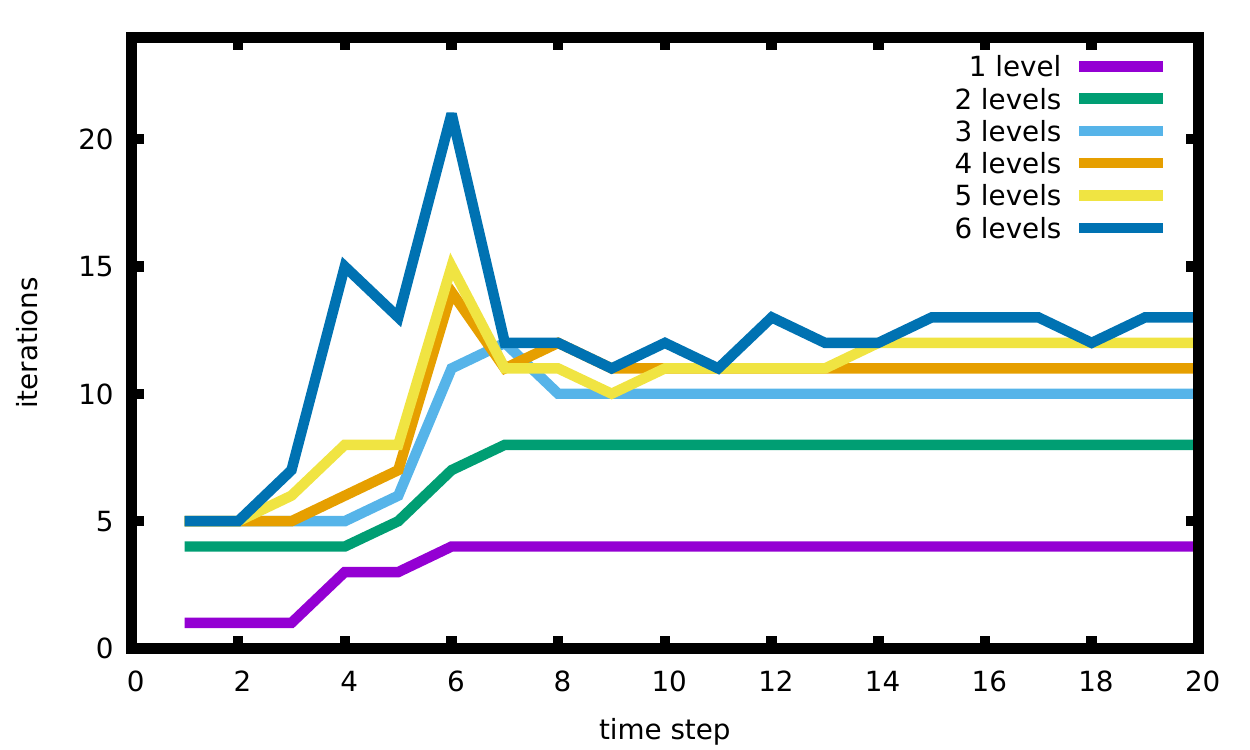}
  \includegraphics[width=0.47\textwidth]{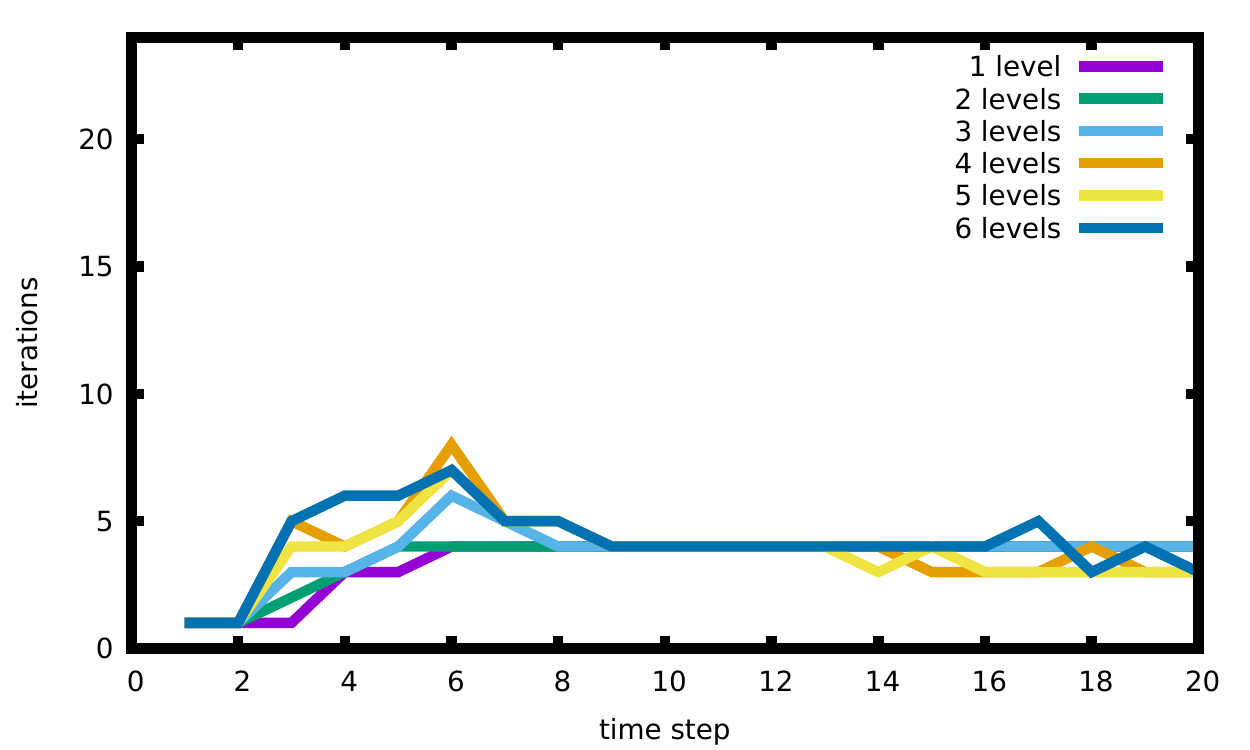}
 \end{center}
 \caption{Iterations per time step: TNNMG (left) vs.\ predictor--corrector (right)}
 \label{fig:iterations_per_time_step}
\end{figure}

We now investigate the convergence behavior of the multigrid solver.  For this, at each time step, we let the solver iterate until
the energy norm of the correction drops below $10^{-7}$.  For each time step the starting iterate
is the zero vector. Then, we use the last iterate $w_* = (u_*, p_*)$ as the reference solution of that
time step, and compute errors of the previous iterates as
\begin{equation*}
 e_\nu^2
 \colonequals
 \norm{u_\nu - u_*}_{H^1}^2 + \norm{p_\nu - p_*}_{L^2}^2,
\end{equation*}
where we have identified the coefficient vectors $u_\nu - u_*$ and $p_\nu - p_*$ with the
corresponding discrete functions.
Figure~\ref{fig:iterations_per_time_step}, left, shows the number of multigrid iterations
needed to make the error $e_\nu$ drop below $10^{-9}$.  The different lines denote the different grid sizes.  As can be seen,
for the most part, this number remains bounded independent from the mesh size, a behavior known from linear multigrid methods
for purely elastic problems.  The behavior is more agitated during time steps 4--6.  This is the time period
where the plastification front sweeps across the domain.  Since here the problems contain a free boundary,
the corresponding time steps are the most challenging ones.

\subsection{Comparing with a predictor--corrector method}

To show the efficiency improvements brought by the TNNMG solver, we now compare it to
the predictor--corrector method with a tangent predictor
and a line search as described in Chapter~\ref{sec:predictor_corrector_methods}.
In the overview chapter~12.2 of~\cite{han_reddy:2013}, it is hailed as the most competitive method
for solving small-strain primal plasticity problems.
We implemented this method in the same code as the TNNMG method, to allow for a fair comparison.
We chose the UMFPack sparse direct solver%
\footnote{\url{http://faculty.cse.tamu.edu/davis/suitesparse.html}}
to solve the elastic predictor problems.
UMFPack is one of the standard open-source direct solvers for general sparse linear systems of equations~\cite{davis:2004}.

We now do precisely the same benchmark test problem as in the previous section, but use the predictor--corrector method.
Figure~\ref{fig:iterations_per_time_step}, right, shows the number of iterations per time step for grids of different
sizes.  Observe that the predictor--corrector method consistently needs about one third of the number of iterations
of the TNNMG method.  Keep in mind, however, that these are {\em predictor--corrector} iterations now, which are
much more expensive than TNNMG iterations.

To quantify the difference, we have measured the run-time per iteration of both methods on grids of different
sizes.
Figure~\ref{fig:walltime_tnnmg_predcorr} shows these run times, \emph{normalized} with the number of degrees
of freedom. We see that for the TNNMG method, these times remain essentially constant. This shows that
multigrid methods can solve algebraic problems with linear time complexity even
in nonlinear situations.  In contrast, normalized wall-times  of the predictor--corrector method go up
with increasing mesh resolution.  On the finest grid, TNNMG is roughly a factor of 40 faster than the
predictor--corrector method, and this ratio will not cease to increase for even larger meshes. The reason
for this is that the predictor-corrector method solves one linear problem at each iteration, whereas
TNNMG performs only a single multigrid iteration.

The predictor--corrector wall-time could be improved by using a linear multigrid for the predictor problem,
rather than the direct sparse solver. In this case the two algorithms would be even more similar to each
other: where TNNMG does one multigrid iteration, the predictor--corrector method would do as many as needed
to solve the predictor problem completely. We have decided no to do this because direct solvers are more common.

\begin{figure}
 \begin{center}
  \includegraphics[width=0.47\textwidth]{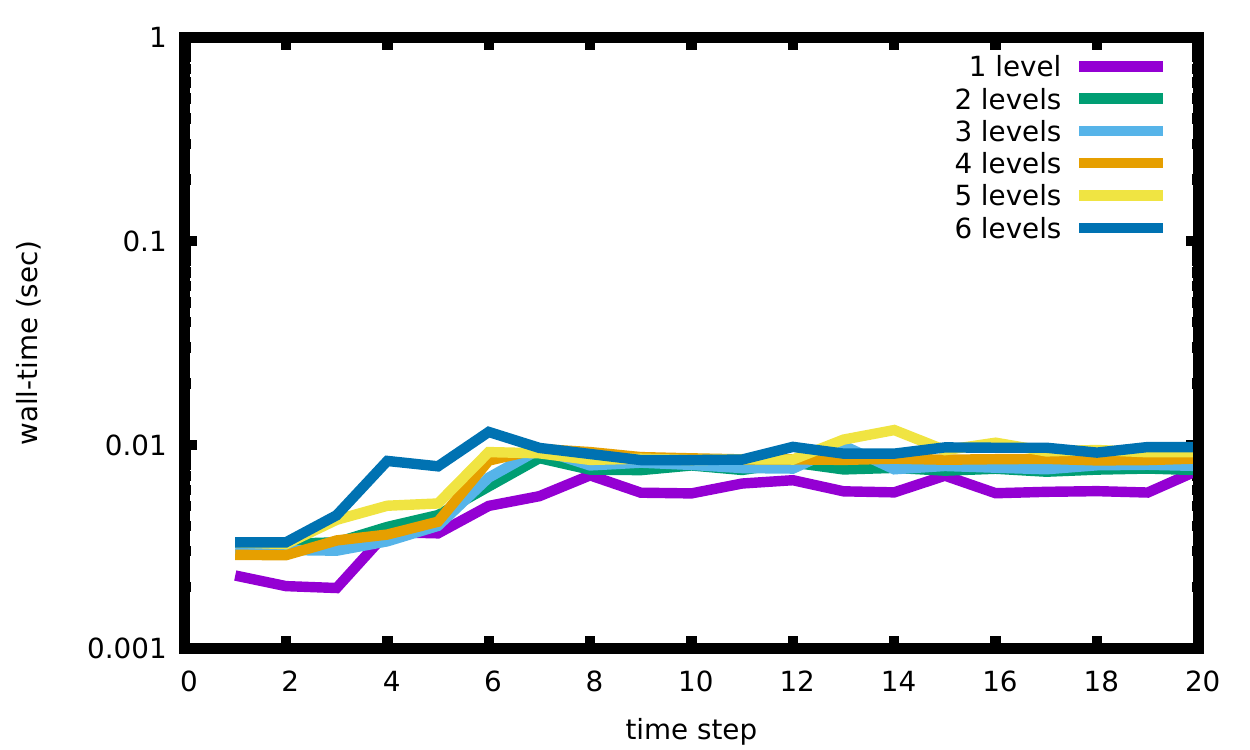}
  \includegraphics[width=0.47\textwidth]{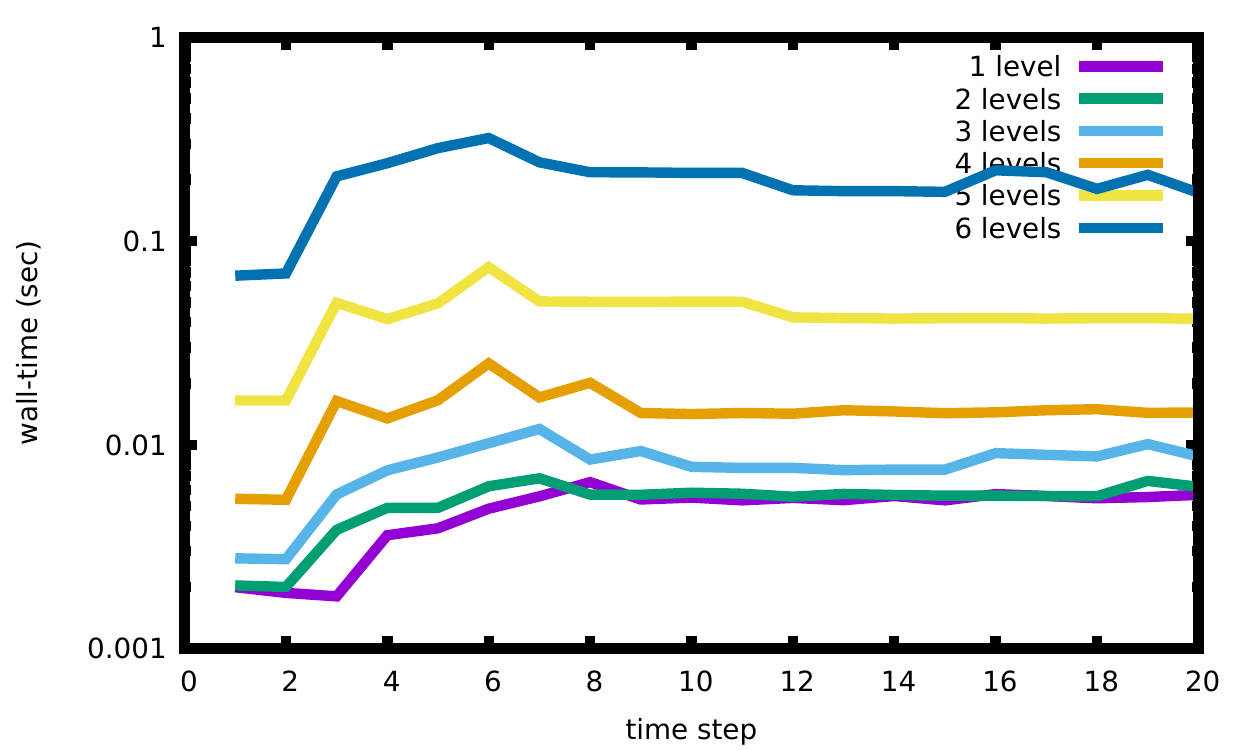}
 \end{center}
 \caption{Normalized wall-time per time step: TNNMG (left) vs.\ predictor--corrector (right)}
 \label{fig:walltime_tnnmg_predcorr}
\end{figure}

\section*{Appendix: The dissipation function for the Tresca yield criterion}

In this appendix we construct the explicit form of the dissipation function for the three-dimensional Tresca yield criterion,
because, while the result is well known, we have not found a proof in the literature.  Remember that the dissipation function $D$ for a given
elastic region $K$ is defined as the support function of $K$
\begin{equation*}
 D : \mathbb{S}^3_0 \to \R \cup \{ \infty \},
 \qquad
 D(\mathbf{p}) \colonequals \sup_{\bs{\sigma} \in K} \{ \mathbf{p} : \bs{\sigma} \}.
\end{equation*}
For the Tresca model, the elastic region $K$ is defined as
\begin{equation*}
 K
 \colonequals
 \{ \bs{\sigma} \in \mathbb{S}^3 \; : \; \max_{i,j=1,\dots,3} \abs{ \sigma_i -\sigma_j } \le \sigma_0 \},
\end{equation*}
where $\sigma_1, \sigma_2, \sigma_3$ are the eigenvalues of $\bs{\sigma}$, and
$\sigma_0 \ge 0$ is the uniaxial yield stress.

\begin{theorem}
\label{thm:tresca_dissipation_function}
 The dissipation function of the Tresca yield criterion is
 \begin{equation*}
  D(\mathbf{p})
  =
  \rho(\mathbf{p})
  \colonequals
  \max \{ \abs{p_1}, \abs{p_2}, \abs{p_3} \},
 \end{equation*}
 where $p_1, p_2, p_3$ are the eigenvalues of $\mathbf{p}$.
\end{theorem}

We first show an intermediate result.
\begin{lemma}
\label{lem:tresca_dissipation_function_intermediate}
For any fixed matrix $\mathbf{p} \in \mathbb{S}^d$ we get
\begin{equation*}
 \sup_{\bs{\sigma} \in K} \{ \mathbf{p} : \bs{\sigma} \}
 =
 \sup_{\sigma_1, \dots, \sigma_d \in \R} \Big\{ \sum_{i=1}^d p_i \sigma_i
 \; : \;
 \max_{i,j=1,\dots,d} \abs{ \sigma_i -\sigma_j } \le \sigma_0 \Big\}.
\end{equation*}
\end{lemma}

\begin{proof}
Let $Q_1$ and $Q_2$ be any two orthogonal matrices.  Then
\begin{equation}
\label{eq:maximal_work_diagonal_1}
 \sup_{\bs{\sigma} \in K} \big\{ Q_1^T \bs{\sigma}Q_1 : Q_2^T\mathbf{p} Q_2 \big\}
 =
 \sup_{\bs{\sigma} \in K} \big\{ Q_2 Q_1^T \bs{\sigma}Q_1 Q_2^T : \mathbf{p} \big\}
 =
 \sup_{\bs{\sigma} \in K} \{ \widetilde{Q}^T \bs{\sigma} \widetilde{Q} : \mathbf{p} \big\},
\end{equation}
where we have written $\widetilde{Q} = Q_1 Q_2^T$.  The definition of $K$ depends only
on the eigenvalues of $\bs{\sigma}$.  Hence $\bs{\sigma}$ is in $K$ if and only if
$\widetilde{Q}^T \bs{\sigma} \widetilde{Q}$ is.  Therefore, the last expression is equal to
\begin{equation}
\label{eq:maximal_work_diagonal_2}
 \sup_{\bs{\sigma} \in K} \{ \widetilde{Q}^T \bs{\sigma} \widetilde{Q} : \mathbf{p} \}
 =
 \sup_{\bs{\sigma} \in K} \{ \bs{\sigma} : \mathbf{p} \}.
\end{equation}
Since this holds for all orthogonal matrices $Q_1$ and $Q_2$, we can pick $Q_1$ such as
to diagonalize $\bs{\sigma}$, and $Q_2$ such as to diagonalize $\mathbf{p}$.  Then
\eqref{eq:maximal_work_diagonal_1} together with \eqref{eq:maximal_work_diagonal_2} is the assertion.
\end{proof}

\bigskip

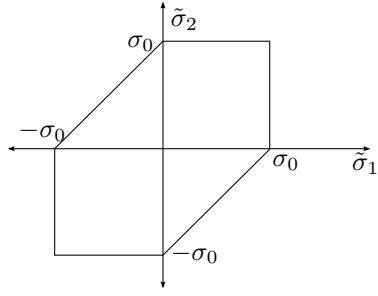
\begin{figure}
\begin{center}
\begin{overpic}[scale=0.5]{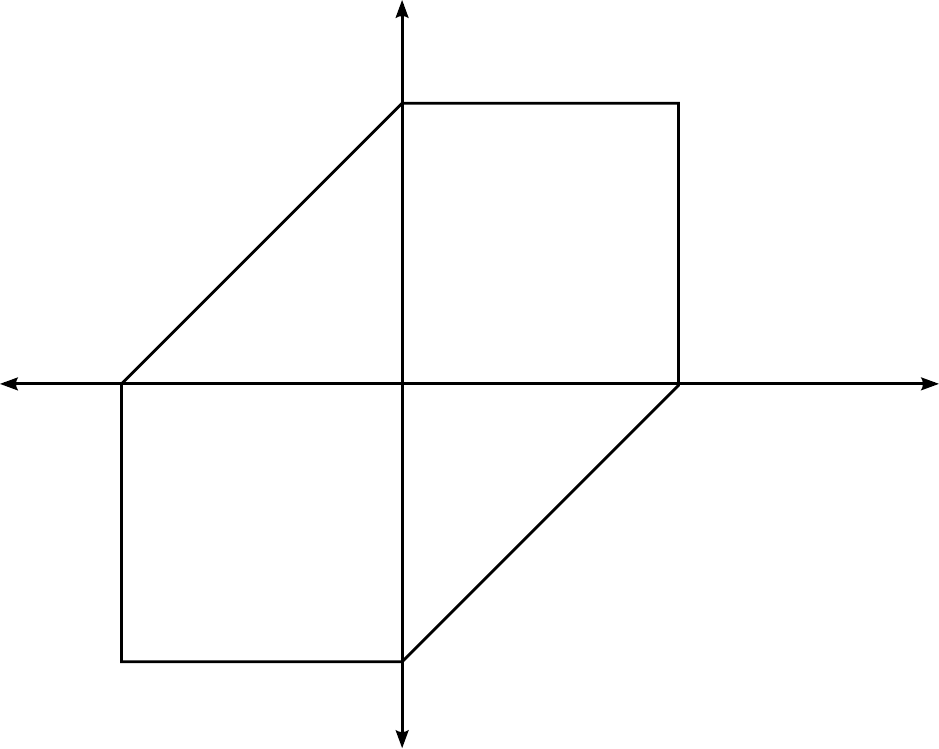}
 \put(33,67){$\sigma_0$}
 \put(73,34){$\sigma_0$}
 \put(45,8){$-\sigma_0$}
 \put(3,42){$-\sigma_0$}
 \put(95,33){$\tilde{\sigma}_1$}
 \put(45,73){$\tilde{\sigma}_2$}
\end{overpic}
\end{center}
\caption{The Tresca hexagon}
\label{fig:tresca_hexagon}
\end{figure}

\begin{proof}[Proof of Theorem~\ref{thm:tresca_dissipation_function}]
By Lemma~\ref{lem:tresca_dissipation_function_intermediate} we have
 \begin{equation*}
  D(\mathbf{p})
  =
  \sup_{\sigma_1,\sigma_2,\sigma_3 \in \R}
  \Big\{p_1 \sigma_1 + p_2 \sigma_2 + p_3 \sigma_3 \; : \; \max_{i,j=1,\dots,3} \abs{\sigma_i - \sigma_j} \le \sigma_0 \Big\}.
 \end{equation*}
We express $\bs{\sigma}$ in new variables
\begin{equation*}
 (\sigma_1, \sigma_2, \sigma_3)
 \mapsto
 \big(\tilde{\sigma}_1 = \sigma_1 - \sigma_3,
      \; \tilde{\sigma}_2 = \sigma_2 - \sigma_3,
      \; \tilde{\sigma}_3 = \sigma_3 - \sigma_3 = 0\big).
\end{equation*}
The scalar product is unchanged under this change of variables.  Indeed,
\begin{align*}
 \sigma_1 p_1 + \sigma_2 p_2 + \sigma_3 p_3
 & =
 (\tilde \sigma_1 + \sigma_3) p_1 + (\tilde \sigma_2 + \sigma_3)p_2 + \sigma_3 p_3 \\
 & =
 \tilde \sigma_1 p_1 + \tilde \sigma_2 p_2 + \sigma_3 (p_1 + p_2 + p_3) \\
 & =
 \tilde \sigma_1 p_1 + \tilde \sigma_2 p_2 \qquad \qquad \quad\, \text{because $\mathbf{p}$ is trace-free}, \\
 & =
 \tilde \sigma_1 p_1 + \tilde \sigma_2 p_2 + \tilde \sigma_3 p_3 \qquad \text{because $\tilde \sigma_3 = 0$}.
\end{align*}
We can therefore write the dissipation function as
\begin{equation*}
 D(\mathbf{p})
 =
  \sup_{\tilde{\sigma}_1,\tilde{\sigma_2},\tilde{\sigma}_3 \in \R}
    \Big\{ \tilde \sigma_1 p_1 + \tilde \sigma_2 p_2 + \tilde \sigma_3 p_3
      \; : \;  \max \{ \abs{\tilde \sigma_1}, \abs{\tilde \sigma_2}, \abs{\tilde \sigma_1 - \tilde \sigma_2 } \} \le \sigma_0,
        \; \tilde \sigma_3 = 0 \Big\}.
\end{equation*}
The maximization problem is now two-dimensional, and the admissible set is the famous Tresca hexagon
(Figure~\ref{fig:tresca_hexagon}).  This hexagon is the convex hull of the six points
\begin{alignat*}{3}
 a_1 &= (\sigma_0,0), &\quad& a_2 = (\sigma_0,\sigma_0), &\quad& a_3 = (0,\sigma_0),\\
 a_4 &= (-\sigma_0,0), & &  a_5 = (-\sigma_0,-\sigma_0), && a_6 = (0,-\sigma_0).
\end{alignat*}
By \cite[page 318]{rockafellar_wets:2010}, we have
\begin{equation*}
 D(\mathbf{p})
 =
 \max \big\{ \langle a_1, (p_1,p_2)\rangle, \dots \langle a_6, (p_1,p_2)\rangle \big\}
\end{equation*}
Hence
\begin{equation*}
 D(\mathbf{p})
 =
 \sigma_0 \max \big\{ p_1, p_1+p_2, p_2, -p_1, -p_1-p_2, -p_2 \big\}.
\end{equation*}
Using that $p_1+p_2+p_3 = 0$ we get
\begin{equation*}
 D(\mathbf{p})
 =
 \sigma_0 \max \big\{ p_1, p_2, p_3, -p_1, -p_2, -p_3 \big\},
\end{equation*}
and hence the assertion.
\end{proof}

\printbibliography

\end{document}